\documentclass{article}

\usepackage{latexsym}
\usepackage{amsmath, amsfonts, amsthm}
\usepackage{epsfig}
\usepackage{stmaryrd}
\usepackage{multicol}
\usepackage{pdfsync}
\usepackage{dsfont}
\usepackage{comment}
\usepackage{algorithmic}
\usepackage{algorithm}
\usepackage{caption}
\usepackage{psfrag}
\usepackage{subfig}
\usepackage{xcolor}
\usepackage{graphics}

\usepackage{mathtools}
\DeclarePairedDelimiter\ceil{\lceil}{\rceil}
\DeclarePairedDelimiter\floor{\lfloor}{\rfloor}

\newcommand{\la}{\left \langle}
\newcommand{\ra}{\right\rangle}

\newcommand{\norm}[1]{\left\lVert #1 \right\rVert}

\newtheorem{theorem}{Theorem}[section]
\newtheorem{corollary}[theorem]{Corollary}
\newtheorem{lemma}[theorem]{Lemma}

\theoremstyle{definition}

\newtheorem{assumption}[theorem]{Assumption}

\theoremstyle{remark}
\newtheorem{remark}[theorem]{Remark}

\numberwithin{equation}{section}

\title{Mean Field Analysis of Neural Networks: A Central Limit Theorem}
\author{Justin Sirignano\footnote{Department of Industrial \& Systems Engineering, University of Illinois at Urbana Champaign, Urbana, E-mail: jasirign@illinois.edu} \phantom{.}  and Konstantinos Spiliopoulos\footnote{Department of Mathematics and Statistics, Boston University, Boston, E-mail: kspiliop@math.bu.edu}
\thanks{K.S. was partially supported by the National Science Foundation (DMS 1550918)}\\
}

\date{\today}

\begin{document}

\maketitle

\begin{abstract}
We rigorously prove a central limit theorem for neural network models with a single hidden layer. The central limit theorem is proven in the asymptotic regime of simultaneously (A) large numbers of hidden units and (B) large numbers of stochastic gradient descent training iterations. Our result describes the neural network's fluctuations around its mean-field limit. The fluctuations have a Gaussian distribution and satisfy a stochastic partial differential equation. The proof relies upon weak convergence methods from stochastic analysis. In particular, we prove relative compactness for the sequence of processes and uniqueness of the limiting process in a suitable Sobolev space.
\end{abstract}

\section{Introduction}


Neural network models have been used as computational tools in many different contexts including machine learning, pattern recognition, physics, neuroscience and statistical mechanics, see for example \cite{Hertz}.
Neural network models, particularly in machine learning, have achieved immense practical success over the past decade in fields such as image, text, and speech recognition. We mathematically analyze neural networks with a single hidden layer in the asymptotic regime of large network sizes and large numbers of stochastic gradient descent iterations. A law of large numbers was previously proven in \cite{NeuralNetworkLLN}, see also \cite{Montanari,Rotskoff_VandenEijnden2018} for related results. This paper rigorously proves a central limit theorem (CLT) for the empirical distribution of the neural network parameters. The central limit theorem describes the fluctuations of the finite empirical distribution of the neural network parameters around its mean-field limit.

The mean-field limit is a law of large numbers for the empirical measure of the neural network parameters as $N \rightarrow \infty$. It satisfies a deterministic nonlinear partial differential equation. The mean-field limit of course is only accurate in the limit $N \rightarrow \infty$, and the central limit theorem provides a first-order correction in $N$. The central limit theorem quantifies the fluctuations of the finite $N$ empirical measure around its mean-field limit. It satisfies a linear stochastic partial differential equation (SPDE) driven by a Gaussian process. In particular, our result shows that the trained neural network behaves as $\mu_t^N \approx \bar \mu_t + \frac{1}{\sqrt{N}} \bar \eta_t$ where $\mu_t^N$ is the empirical measure of the parameters for a neural network with $N$ hidden units, $\bar \mu_t$ is the mean-field limit, and $\bar \eta_t$ is the Gaussian correction from the central limit theorem.

The proof requires a linearization of the nonlinear pre-limit evolution equation for the empirical distribution of the neural network parameters. This linearization produces several remainder terms which must be shown to vanish in the limit (similar to a perturbation analysis for PDEs). The SPDE for the CLT $\bar \eta_t$ is linearized around the nonlinear PDE for the mean-field limit $\bar \mu_t$. The CLT SPDE and mean-field limit PDE are therefore coupled. We must also show that the pre-limit evolution equation (which is in discrete time since stochastic gradient descent is a discrete-time algorithm) converges to a continuous-time limit.

The proof relies upon weak convergence analysis for interacting particle systems. The convergence analysis is technically challenging since the fluctuations of the empirical distribution is a signed-measure-valued process and its limit process turns out to be distribution-valued in the appropriate space. Unfortunately, the space of signed measures endowed with the weak topology is in general not metrizable (see \cite{DelBarrio} and \cite{FluctuationSpiliopoulosSirignanoGiesecke} for further discussion of the space of signed measures). We study the convergence of the fluctuations as a process taking values in the dual space of an appropriate Sobolev space. We prove that the pre-limit fluctuation process is relatively compact in that space and that any limit point is unique in that space. In particular, we will use the dual space $W^{-J,2}=W^{-J,2}(\Theta)$ of the Sobolev space $W_0^{J,2}(\Theta)$ with $\Theta$ a bounded subset of the appropriate Euclidean space and where $J$ is sufficiently large; see Section \ref{WeightedSobolevSpace} for a detailed description. Since the pre-limit evolution equation has discrete updates, we study convergence in the Skorokhod space $D_{W^{-J,2}}([0,T])$. ($D_S([0,T])$ is the set of maps from $[0,T]$ into $S$ which are right-continuous and which have left-hand limits.)

Most of the literature on central limit theorems for interacting particle systems considers continuous-time systems, see for example \cite{Meleard,KurtzXiong,FluctuationSpiliopoulosSirignanoGiesecke,Chevallier,Dawson,CometsEisele}. In contrast, in this article the pre-limit process is in discrete time and converges to a continuous-time limit process after an appropriate time rescaling. At a practical level, this shows that the relation between the number of particles (``hidden units" in the language of neural networks) and the number of stochastic gradient steps should be of the same order to have convergence and statistically good behavior. At a more mathematical level, this passage from discrete to continuous time produces a number of additional remainder terms that must be shown to vanish at the correct rate in order for a CLT to hold. We resolve all these issues for one-layer neural network models, rigorously establishing and characterizing the fluctuations limit.

Weak convergence and mean field analysis has been used in many other disciplines, including interacting particle systems in physics, neural networks in biology and financial modeling, see for example \cite{TypicalDefaults}, \cite{LargePortfolio}, \cite{DaiPra1}, \cite{DaiPra2}, \cite{DaiPra3}, \cite{Capponi}, \cite{Hambly}, \cite{Delarue}, \cite{Inglis}, \cite{Moynot}, \cite{Touboul},  \cite{Sompolinsky} and the references therein for a certainly not-complete list.
 Recently, \cite{NeuralNetworkLLN}, \cite{Mattingly}, \cite{Montanari}, and \cite{Rotskoff_VandenEijnden2018} study mean-field limits of machine learning algorithms, including neural networks. In this paper, we rigorously establish a central limit theorem for neural networks trained with stochastic gradient descent. \cite{Rotskoff_VandenEijnden2018} also formally studies corrections to the mean field limit.

Consider the one-layer neural network
\begin{eqnarray}
g_{\theta}^N(x) = \frac{1}{N} \sum_{i=1}^N c^i \sigma( w^i \cdot x),\label{Eq:NN}
\end{eqnarray}
where for every $i\in\{1,\cdots, N\}$, $c^i \in \mathbb{R}$ and $x, w^i \in \mathbb{R}^{d}$. For notational convenience we shall interpret $w^i \cdot x=\sum_{j=1}^{d}w^{i,j} x^{j}$ as the standard scalar inner product. The neural network model has parameters $\theta = (c^1, \ldots, c^N, w^1, \ldots, w^N)\in\mathbb{R}^{(1+d)N}$, which must be estimated from data.

The neural network (\ref{Eq:NN}) takes a linear function of the original data, applies an element-wise nonlinear operation using the function $\sigma: \mathbb{R} \rightarrow \mathbb{R}$, and then takes another linear function to produce the output. The activation function $\sigma(\cdot)$ is a nonlinear function such as a sigmoid or tanh function. The quantity $\sigma( w^i \cdot x)$ is referred to as the $i$-th ``hidden unit", and the vector $\big{(} \sigma( w^1 \cdot x), \ldots, \sigma( w^N \cdot x) \big{)}$ is called the ``hidden layer". The number of units in the hidden layer is $N$.

The objective function is
\begin{eqnarray}
L(\theta) = \mathbb{E}_{Y,X} [ ( Y - g_{\theta}^N(X) )^2 ],\label{Eq:ObjFunction}
\end{eqnarray}
where the data $(Y,X)$ is assumed to have a  joint distribution $\pi (dx,dy)$. We shall write $\mathcal{X}\subset\mathbb{R}^{d}$ and $\mathcal{Y}\subset\mathbb{R}$ for the state spaces of $X$ and $Y$, respectively. The parameters $\theta = (c^1, \ldots, c^N, w^1, \ldots, w^N)$ are estimated using stochastic gradient descent:
\begin{eqnarray}
c_{k+1}^i &=& c^i_k + \frac{\alpha}{N} (y_k - g_{\theta_k}^N(x_k) )  \sigma (w^i_k \cdot x_k), \notag \\
w^{i,j}_{k+1} &=& w^{i,j}_k + \frac{\alpha}{N} (y_k - g_{\theta_k}^N(x_k) )  c^i_k \sigma' (w^i_k \cdot x_k) x^{j}_k, \quad j=1,\cdots, d,\label{Eq:SGD}
\end{eqnarray}
where $\alpha$ is the learning rate and $(x_k, y_k) \sim \pi(dx,dy)$. Stochastic gradient descent minimizes (\ref{Eq:ObjFunction}) using a sequence of noisy (but unbiased) gradient descent steps $\nabla_{\theta} [ ( y_k - g_{\theta_k}^N(x_k) )^2 ]$. 
Stochastic gradient descent typically converges more rapidly than gradient descent for large datasets. For this reason, stochastic gradient descent is widely used in machine learning.

Define the empirical measure
\begin{eqnarray}
\nu^N_k(dc, dw) = \frac{1}{N} \sum_{i=1}^N \delta_{c_k^i, w_k^i}(dc, dw).\notag
\end{eqnarray}

The neural network's output can be re-written in terms of the empirical measure:
\begin{eqnarray}
g_{\theta_k}^N(x)  = \la c \sigma(w \cdot x),  \nu^N_k \ra.\notag
\end{eqnarray}
$\la f, h \ra$ denotes the inner product of $f$ and $h$. For example, $ \la c \sigma(w \cdot x),  \nu^N_k \ra = \int c \sigma(w \cdot x) \nu^N_k(dc, dw)  $.

The scaled empirical measure is
\begin{eqnarray}
\mu^N_t = \nu^N_{\floor*{N t} }.\notag
\end{eqnarray}

The scaled empirical measure $\mu^N$ is a random element of the Skorokhod space $D_{E}([0,T])$\footnote{$D_S([0,T])$ is the set of maps from $[0,T]$ into $S$ which are right-continuous and which have left-hand limits.} with $E = \mathcal{M}(\mathbb{R}^{1+d})$.

We shall work on a filtered probability space $(\Omega,\mathcal{F},\mathbb{P})$  on which all the random variables are defined. The probability space is equipped with a filtration 
 that is right continuous and contains all $\mathbb{P}$-null sets.

We impose the following conditions.
\begin{assumption} \label{A:Assumption1} We have that
\begin{itemize}
\item The activation function $\sigma\in C^{\infty}_{b}(\mathbb{R})$.
\item The data $(X,Y)\in \mathcal{X}\times\mathcal{Y}$ is compactly supported.
\item The sequence of data samples $(x_k, y_k)$ is i.i.d.
\item The random initialization $(c_0^i, w_0^i)$ is i.i.d, generated from a distribution $\bar{\mu}_{0}$ with compact support.
\end{itemize}
\end{assumption}

\subsection{Law of Large Numbers} \label{LLNintro}

\cite{NeuralNetworkLLN} proves the mean-field limit $\mu^N \overset{p} \rightarrow \bar \mu$ as $N \rightarrow \infty$. The convergence theorems of \cite{NeuralNetworkLLN} are summarized below.

\begin{theorem} \label{TheoremLLN}
 Assume Assumption \ref{A:Assumption1}. The scaled empirical measure $\mu^N_t$ converges in distribution to $\bar \mu_t$ in $D_E([0,T])$ as $N \rightarrow \infty$. For every $f\in C^{2}_{b}(\mathbb{R}^{1+d})$, $\bar \mu$ is the deterministic unique solution of the measure evolution equation
\begin{eqnarray}
\la f, \bar \mu_t \ra  &=& \la f, \bar \mu_0 \ra + \int_0^t   \bigg{(} \int_{\mathcal{X}\times\mathcal{Y}}   \alpha \big{(} y -  \la c' \sigma(w'\cdot x),  \bar \mu_s \ra \big{)} \la \nabla (c\sigma(w \cdot x)) \cdot \nabla f, \bar \mu_s \ra   \pi(dx,dy)\bigg{)} ds,
\label{EvolutionEquationIntroduction}
\end{eqnarray}
where $\nabla f=(\partial_{c}f,\nabla_{w}f)$.
\end{theorem}

\begin{remark}\label{R:ProbConvergence}
Since weak convergence to a constant implies convergence in probability, Theorem \ref{TheoremLLN} leads to the stronger result of convergence in probability
\begin{equation*}\label{E:mulimit}
  \lim_{N\to \infty}\mathbb{P}\left\{ d_{E}(\mu^N,\bar \mu)\ge \delta\right\} = 0
  \end{equation*}
for every $\delta>0$ and where $d_E$ is the metric for $D_E([0,T])$.
\end{remark}
\begin{corollary} \label{CorollaryLLN}
 Assume Assumption \ref{A:Assumption1}.  Suppose that $\bar \mu_0$ admits a density $p_0(c,w)$ and that there exists a unique solution to the nonlinear partial differential equation
\begin{eqnarray}
\frac{ \partial p(t, c, w) }{ \partial t}  &=& - \alpha \int_{\mathcal{X}\times\mathcal{Y}}   \bigg{(} \big{(} y -  \la c' \sigma(w' \cdot x), p(t,c', w') \ra \big{)} \frac{\partial}{\partial c} \big{[} \sigma(w \cdot x) p(t,c,w) \big{]} \bigg{)}\pi(dx,dy)  \notag \\
& &\quad- \alpha \int_{\mathcal{X}\times\mathcal{Y}}   \bigg{(} \big{(} y -  \la c' \sigma(w' \cdot x),  p(t,c',w') \ra \big{)} x \cdot \nabla_{w} \big{[}  c \sigma'(w \cdot x)  p(t,c,w) \big{]} \bigg{)}\pi(dx,dy), \notag \\
p(0,c,w) &=& p_0(c,w).\notag
\end{eqnarray}
Then, we have that the solution to the measure evolution equation (\ref{EvolutionEquationIntroduction}) is such that
\begin{eqnarray*}
\bar \mu_t(dc, dw) = p(t, c,w) dc dw. \notag
\end{eqnarray*}
\end{corollary}

\subsection{Main Result: A Central Limit Theorem} \label{CLTintro}

In this paper, we prove a central limit theorem for one-layer neural networks as the size of the network and the number of training steps become large. The central limit theorem quantifies the speed of convergence of the finite neural network to its mean-field limit as well as how the finite neural network fluctuates around the mean-field limit for large $N$.

Define the fluctuation process
\begin{eqnarray}
\eta_t^N = \sqrt{N} ( \mu_t^N - \bar \mu_t ).\notag
\end{eqnarray}

We prove that $\eta^N \overset{d} \rightarrow \bar \eta$, where $\bar \eta$ satisfies a stochastic partial differential equation. This result characterizes the fluctuations of the finite empirical measure $\mu^N$ around its mean-field limit $\bar \mu$ for large $N$. The limit $\bar \eta$ has a Gaussian distribution. We study the convergence of $\eta_t^N$ in the space $D_{W^{-J,2} }([0,T])$, where $W^{-J,2}=W^{-J,2}(\Theta)$ is the dual of the Sobolev space $W_0^{J,2}(\Theta)$ with $\Theta\subset \mathbb{R}^{1+d}$ a  bounded domain. These spaces are described in detail in Section \ref{WeightedSobolevSpace}.

\begin{theorem} \label{MainTheoremCLT}
Assume Assumption \ref{A:Assumption1} and let $J \geq 3  \ceil*{\frac{d+1}{2}} + 7$.  Let $0<T<\infty$ be given. The sequence $\{\eta^{N}_{t},t\in[0,T]\}_{N\in\mathbb{N}}$ is relatively compact in $D_{W^{-J,2}}([0,T])$. The sequence of processes $\{\eta^{N}_{t},t\in[0,T]\}_{N\in\mathbb{N}}$  converges in distribution in $D_{W^{-J,2}}([0,T])$ to the process $\{\bar{\eta}_{t},t\in[0,T]\}$, which, for every $f \in W_0^{J,2}(\Theta)$,  satisfies the stochastic partial differential equation
\begin{eqnarray}
\la f, \bar \eta_t \ra  &=& \la f, \bar \eta_0 \ra +    \int_0^t \int_{\mathcal{X}\times\mathcal{Y}} \alpha \big{(} y -  \la c \sigma(w \cdot x), \bar \mu_{s} \ra \big{)} \la \nabla(c\sigma(w \cdot x)) \cdot \nabla f, \bar \eta_{s} \ra \pi(dx,dy)  ds \notag \\
& &-  \int_0^t   \int_{\mathcal{X}\times\mathcal{Y}}  \alpha \la c \sigma(w \cdot x), \bar \eta_s \ra \la \nabla(c\sigma(w \cdot x))\cdot \nabla f, \bar \mu_{s} \ra \pi(dx,dy)  ds +\la f, \bar M_t \ra.
\label{SPDEmain}
\end{eqnarray}

$\bar M_t$ is a mean-zero Gaussian process; see Lemma \ref{VarianceOfBarM} for its covariance structure. Finally, the stochastic evolution equation (\ref{SPDEmain}) has a unique solution in $W^{-J,2}$, which implies that $\bar{\eta}$ is unique.
\end{theorem}

The CLT SPDE (\ref{SPDEmain}) is coupled with the mean-field limit PDE (\ref{EvolutionEquationIntroduction}). (\ref{EvolutionEquationIntroduction}) is a deterministic nonlinear PDE while (\ref{SPDEmain}) is a stochastic linear PDE. The SPDE (\ref{SPDEmain}) is linear in $\bar{\eta}$ and driven by a Gaussian process; therefore, the limiy $\bar \eta_t$ itself is a Gaussian process.

Theorem \ref{MainTheoremCLT} indicates that for large $N$ the empirical distribution of the neural network's parameters behaves as
\begin{eqnarray}
\nu^{N}_{\floor*{N \cdot} }=\mu^{N}_{\cdot}\approx \bar{\mu}_{\cdot}+\frac{1}{\sqrt{N}}\bar{\eta}_{\cdot},\notag
\end{eqnarray}
where $\bar \eta$ has a Gaussian distribution. Combined, Theorems \ref{TheoremLLN}  and \ref{MainTheoremCLT}  show that the relation between the number of particles ("hidden units" in the language of neural networks) and the number of stochastic gradient steps should be of the same order to have convergence and statistically good behavior. Under this scaling, as a measure valued process, the empirical distribution of the parameters behaves as a Gaussian distribution with specific variance-covariance structure (as indicated by Theorem \ref{MainTheoremCLT}).

\subsection{Outline of Paper} \label{Outline}
In Section \ref{WeightedSobolevSpace} we present the Sobolev spaces with respect to which  convergence is studied. The pre-limit evolution equation for the fluctuation process $\eta^N$ is derived in Section \ref{PreliminaryCalculations}. Section \ref{RelativeCompactness} proves relative compactness. Section \ref{Identification} derives the limiting SPDE (\ref{SPDEmain}). Uniqueness of the SPDE (\ref{SPDEmain}) is proven in Section \ref{Uniqueness}. Section \ref{ProofOfMainResult} collects these results and proves Theorem \ref{MainTheoremCLT}. Conclusions are in Section \ref{Conclusion}.

\section{Sobolev Spaces} \label{WeightedSobolevSpace}

We study convergence in a Sobolev space \cite{Adams}. Weighted Sobolev spaces have been previously used to study central limit theorems of mean field systems in papers such as \cite{Meleard}, \cite{KurtzXiong} and \cite{FluctuationSpiliopoulosSirignanoGiesecke}. Weights are not necessary in this paper since $\eta_t^N$ and $\mu_t^N$ are compactly supported uniformly with respect to $N\in\mathbb{N}$ and $t\in[0,T]$ (see Lemma \ref{CompactLemmaXi}).

Let $\Theta \subset \mathbb{R}^D$ be a bounded domain with $D = d + 1$. For any integer $J\in\mathbb{N}$, consider the space of real valued functions $f$ with partial derivatives up to order $J$ which satisfy
\begin{eqnarray}
\norm{f}_J = \bigg{(} \sum_{|k| \leq J} \int_{\Theta}  \big{|}  D^k f(x) \big{|}^2 dx\bigg{)}^{1/2} < \infty.\notag
\end{eqnarray}

Define  the  space $W_0^{J,2}(\Theta)$ as the closure of functions of class $C_0^{\infty}(\Theta)$ in the norm defined above. $C_0^{\infty}(\Theta)$ is the space of all functions in $C^{\infty}(\Theta)$ with compact support. (The space $W_0^{J,2}(\Theta)$ is frequently also denoted by $H^{J}_{0}(\Theta)$ in the literature.) $W_0^{J,2}(\Theta)$ is a Hilbert space (see Theorem 3.5 and Remark 3.33 in \cite{Adams}) and has the inner product
\begin{eqnarray}
\la f, g \ra_{J} =  \sum_{|k| \leq J} \int_{\Theta} D^k f(x) D^k g(x) dx.\notag
\end{eqnarray}

When $J=0$, we write $\la f, g \ra_{0}=\la f, g \ra$. $W^{-J,2}(\Theta)$ denotes the dual space of $W_0^{J,2}(\Theta)$ that is equipped with the norm
\begin{eqnarray}
\norm{f}_{-J} = \sup_{g \in W_0^{J,2}(\Theta)} \frac{ \big{|} \la f, g \ra \big{|} }{ \norm{g}_J}.\notag
\end{eqnarray}

We will study convergence in the Sobolev space corresponding to $J \geq 3 \ceil*{\frac{D}{2}} + 7$. From Lemma \ref{CompactLemmaXi}, we have that $\mu_t^N$ and $\eta_t^N$ are compactly supported. In particular, there exists a compact set $K=[-C_{o},C_{o}]^{D}\subset \mathbb{R}^{D}$  such that $\mu_t^N$ and $\eta_t^N$ vanish outside the compact set $K$ for every $N\in\mathbb{N}$ and $t\in[0,T]$.  We choose $\Theta = (-B, B)^{D}$ where $B =  3 \sqrt{D} C_{o}$. Note that $C_{o}$, and thus the domain $\Theta$, may depend upon fixed parameters of the problem such that $T$, $\alpha$, $\pi(dx,dy)$, and $\bar \mu_0$, but what is important is that the bounded set $\Theta$ is fixed and does not change with $N\in\mathbb{N}$ or $t\in[0,T]$.

Sometimes, we may write for simplicity $W^{-J,2}$ in place of $W^{-J,2}(\Theta)$ and $W_0^{J,2}$ in place of $W_0^{J,2}(\Theta)$.

\section{Preliminary Calculations} \label{PreliminaryCalculations}

The goal of this section is to write $\la f, \eta^{N}_{t}\ra$, with $\eta^{N}_{t}$ being the fluctuation process and $f \in C^{2}_{b}(\mathbb{R}^{1+d})$ a test function, in a way that allows us to take limits. In particular, our goal is to describe the evolution of $\la f, \eta^{N}_{t}\ra$ in terms of the equation (\ref{EtaEqn1}). In order to do this, we need some preliminary computations.

We consider the evolution of the empirical measure $\nu^N_k$ via test functions $f \in C^{2}_{b}(\mathbb{R}^{1+d})$. A Taylor expansion yields
\begin{align}
\la f , \nu^N_{k+1} \ra - \la f , \nu^N_k \ra &= \frac{1}{N} \sum_{i=1}^N \bigg{(} f(c^i_{k+1}, w^i_{k+1} ) -  f(c^i_{k}, w^i_{k} )  \bigg{)} \notag \\
&= \frac{1}{N} \sum_{i=1}^N \partial_c f(c^i_{k}, w^i_{k} ) ( c^i_{k+1} -  c^i_{k} )  + \frac{1}{N} \sum_{i=1}^N \nabla_w  f(c^i_{k}, w^i_{k} )^{\top}  ( w^i_{k+1} -  w^i_{k} ) \notag \\
& \quad +\frac{1}{N} \sum_{i=1}^N \partial^{2}_{c} f(\bar c^i_{k},  \bar w^i_{k} ) ( c^i_{k+1} -  c^i_{k} )^2  + \frac{1}{N} \sum_{i=1}^N ( c^i_{k+1} -  c^i_{k} )\nabla_{cw}  f(\bar c^i_{k}, \bar w^i_{k} )( w^i_{k+1} -  w^i_{k} )    \notag \\
&  \quad+\frac{1}{N} \sum_{i=1}^N ( w^i_{k+1} -  w^i_{k} )^{\top}\nabla^{2}_{w} f(\bar c^i_{k}, \bar w^i_{k} ) ( w^i_{k+1} -  w^i_{k} ),\label{FirstEqn}
\end{align}
for points $\bar c^{i}_{k}, \bar w^{i}_{k}$ in the segments connecting $c^i_{k+1}$ with $c^i_{k}$ and  $w^i_{k+1}$ with $w^i_{k}$, respectively.  Under the compactness part of Assumption \ref{A:Assumption1}, the results of \cite{NeuralNetworkLLN} imply that the parameters are uniformly bounded (in both $0 \leq k \leq N T$ and $N$):
\begin{eqnarray}
| c_k^i | + \norm{w_k^i} < C_{o}.
\label{UniformBoundfromLLNpaper}
\end{eqnarray}

We shall also denote by  $\mathcal{F}_k^N$ to be the $\sigma-$algebra generated by $(c^{i}_{0},w^{i}_{0})_{i=1}^{N}$ and $(x_{j}, y_{j})_{j=0}^{k-1}$.
Using the relation (\ref{Eq:SGD}), equation (\ref{FirstEqn}) becomes
\begin{eqnarray}
\la f , \nu^N_{k+1} \ra - \la f , \nu^N_k \ra &=&  \frac{1}{N^2} \sum_{i=1}^N \partial_c f(c^i_{k}, w^i_{k} )  \alpha (y_k - g_{\theta_k}^N(x_k) )  \sigma (w^i_k \cdot x_k)   \notag \\
& &+ \frac{1}{N^2} \sum_{i=1}^N   \alpha (y_k - g_{\theta_k}^N(x_k) )  c^i_k \sigma' (w^i_k \cdot x_k) \nabla_w  f(c^i_{k}, w^i_{k} )\cdot x_{k} + \frac{G_k^N}{N^2}. \notag
\end{eqnarray}
where $\frac{G_k^N}{N^2}$ is an $O\left(N^{-2}\right)$ term with
\begin{eqnarray}
G_k^N &=& N^2 \bigg{(} \frac{1}{N} \sum_{i=1}^N \partial^{2}_{c} f(\bar c^i_{k},  \bar w^i_{k} ) ( c^i_{k+1} -  c^i_{k} )^2  + \frac{1}{N} \sum_{i=1}^N ( c^i_{k+1} -  c^i_{k} )\nabla_{cw}  f(\bar c^i_{k}, \bar w^i_{k} )( w^i_{k+1} -  w^i_{k} )    \notag \\
& &+ \frac{1}{N} \sum_{i=1}^N ( w^i_{k+1} -  w^i_{k} )^{\top}\nabla^{2}_{w} f(\bar c^i_{k}, \bar w^i_{k} ) ( w^i_{k+1} -  w^i_{k} ) \bigg{)}.\notag
\end{eqnarray}
Note that $|G_k^N| < C \displaystyle \sum_{|\alpha| =2 } \sup_{c,w \in K} |D^{\alpha} f(c,w)  |$ due to the uniform bound $| c_k^i| + \norm{w_k^i} < C_{o}$, $(X,Y)$ having compact support, and the relation (\ref{Eq:SGD}). $K \subset \mathbb{R}^{1+d}$ is the compact set $K=[-C_{o},C_{o}]^{1+d}$.

We next define the following components:
\begin{eqnarray}
D^{1,N}_k &=& \frac{1}{N} \int_{\mathcal{X}\times\mathcal{Y}}   \alpha \big{(} y -  \la c \sigma(w \cdot x),  \nu^N_k \ra \big{)} \la \sigma(w \cdot x) \partial_c f, \nu_k^N \ra \pi(dx,dy), \notag \\
D^{2,N}_k &=&  \frac{1}{N} \int_{\mathcal{X}\times\mathcal{Y}}   \alpha \big{(} y -  \la c \sigma(w \cdot x),  \nu^N_k \ra \big{)} \la c  \sigma'(w \cdot x) x \cdot \nabla_w f, \nu_k^N \ra  \pi(dx,dy), \notag \\
\la f, M^{1,N}_k  \ra &=&  \frac{1}{N} \alpha \big{(} y_k -  \la c \sigma(w \cdot x_k),  \nu^N_k \ra \big{)} \la \sigma(w \cdot x_k) \partial_c f, \nu_k^N \ra   - D^{1,N}_k, \notag \\
\la f, M^{2,N}_k \ra &=& \frac{1}{N} \alpha \big{(} y_k -  \la c \sigma(w \cdot x_k),  \nu^N_k \ra \big{)} \la c  \sigma'(w \cdot x_k) x \cdot \nabla_w f, \nu_k^N \ra  - D^{2,N}_k.\notag
\end{eqnarray}

Combining the different terms together, we subsequently obtain
\begin{eqnarray*}
\la f , \nu^N_{k+1} \ra - \la f , \nu^N_k \ra &=&  D^{1,N}_k + D^{2,N}_k + \la f,  M^{1,N}(t) \ra + \la f,  M^{2,N}(t) \ra + O\left(N^{-2}\right).
\end{eqnarray*}

Next, we define the scaled versions of $D^{1,N}, D^{2,N}, M^{1,N}$ and $M^{2,N}$:
\begin{eqnarray}
D^{1,N}(t) &=&  \sum_{k=0}^{ \floor*{N t}-1 } D^{1,N}_k, \qquad D^{2,N}(t) =  \sum_{k=0}^{ \floor*{N t} -1} D^{2,N}_k , \notag \\
\la f, M^{1,N}(t) \ra &=& \sum_{k=0}^{ \floor*{N t} -1 } \la f, M^{1,N}_k \ra, \qquad \la f, M^{2,N}(t) \ra = \sum_{k=0}^{ \floor*{N t} -1} \la f, M^{2,N}_k \ra. \notag
\end{eqnarray}

As it will be demonstrated in Section \ref{SS:CompactContainmentM}, $\la f, M^{1,N}(t) \ra$ and $\la f, M^{2,N}(t) \ra$ are martingale terms. We also define
\begin{eqnarray}
\la f, M^{N}_t  \ra  = \la f,  M^{1,N}(t) \ra + \la f, M^{2,N}(t) \ra .\notag
\end{eqnarray}

$D^{1,N}(t)$ and $D^{2,N}(t)$ can be approximated by integrals:
\begin{eqnarray}
\sum_{k=0}^{ \floor*{Nt}-1} D^{1,N}_k &=& \sum_{k=0}^{ \floor*{Nt}-1}  \int_{ \frac{k}{N}}^{ \frac{k+1}{N}} \int_{\mathcal{X}\times\mathcal{Y}}   \alpha \big{(} y -  \la c \sigma(w \cdot x),  \nu^N_k \ra \big{)} \la \sigma(w \cdot x) \partial_c f, \nu_k^N \ra \pi(dx,dy) ds \notag \\
&=&  \sum_{k=0}^{ \floor*{Nt}-1}  \int_{ \frac{k}{N}}^{ \frac{k+1}{N}} \int_{\mathcal{X}\times\mathcal{Y}}   \alpha \big{(} y -  \la c \sigma(w \cdot x),  \mu^N_s \ra \big{)} \la \sigma(w \cdot x) \partial_c f, \mu_s^N \ra \pi(dx,dy) ds \notag \\
&=&  \int_{ 0}^t \int_{\mathcal{X}\times\mathcal{Y}}   \alpha \big{(} y -  \la c \sigma(w \cdot x),  \mu^N_s \ra \big{)} \la \sigma(w \cdot x) \partial_c f, \mu_s^N \ra \pi(dx,dy) ds + V^{1,N}_t,\notag
\end{eqnarray}
where $V^{1,N}_t$ is a remainder term defined below. Similarly,
\begin{eqnarray}
\sum_{k=0}^{ \floor*{Nt}-1} D^{2,N}_k =  \int_{ 0}^t \int_{\mathcal{X}\times\mathcal{Y}}   \alpha \big{(} y -  \la c \sigma(w \cdot x),  \mu^N_s \ra \big{)} \la c \sigma'(w \cdot x) x \cdot \nabla_w f, \mu_s^N \ra \pi(dx,dy) ds + V^{2,N}_t.\notag
\end{eqnarray}

The remainder terms $V^{1,N}_t$ and $V^{2,N}_t$ are
\begin{eqnarray}
V^{1,N}_t &=&  - \int_{ \frac{ \floor*{Nt} }{N} }^t \int_{\mathcal{X}\times\mathcal{Y}}   \alpha \big{(} y -  \la c \sigma(w \cdot x),  \mu^N_s \ra \big{)} \la \sigma(w \cdot x) \partial_c f, \mu_s^N \ra \pi(dx,dy) ds, \notag \\
V^{2,N}_t &=& - \int_{ \frac{ \floor*{Nt} }{N} }^t \int_{\mathcal{X}\times\mathcal{Y}}   \alpha \big{(} y -  \la c \sigma(w \cdot x),  \mu^N_s \ra \big{)} \la c \sigma'(w \cdot x) x \cdot \nabla_w f, \mu_s^N \ra \pi(dx,dy) ds, \notag \\
V^N_t &=& V^{1,N}_t + V^{2,N}_t.\notag
\end{eqnarray}

$V_t^N$ is a c\'adl\'ag process with jumps at times $ \frac{1}{N}, \frac{2}{N}, \ldots, \frac{\floor*{NT} }{N}$. Furthermore, due to the uniform bound (\ref{UniformBoundfromLLNpaper}) and $\mathcal{X} \times \mathcal{Y}$ being a compact set, $V_t^N$ is an $\mathcal{O}( N^{-1})$ remainder term:
\begin{eqnarray}
\sup_{ t \in [0,T]} | V_t^N | \leq \frac{C}{N}  \sum_{|\alpha| =1 } \sup_{c,w \in K} |D^{\alpha} f(c,w)  |
\label{Vbound}
\end{eqnarray}

The scaled empirical measure can be written as the telescoping sum
\begin{eqnarray}
\la f, \mu^N_{t} \ra - \la f, \mu^N_{0} \ra
&=& \la f, \nu^N_{\floor*{N t} } \ra - \la f, \nu^N_{0} \ra \notag \\
&=& \bigg{(} \la f , \nu^N_{  \floor*{Nt} } \ra - \la f , \nu^N_{  \floor*{Nt}-1 } \ra \bigg{)} + \bigg{(} \la f , \nu^N_{  \floor*{Nt} -1} \ra - \la f , \nu^N_{  \floor*{Nt}-2 } \ra \bigg{)}  \notag \\
& &+ \ldots +  \bigg{(} \la f , \nu^N_{  1} \ra - \la f , \nu^N_{  0} \ra \bigg{)} \notag \\
&=&\sum_{k=0}^{ \floor*{Nt}-1}  \bigg{(} \la f , \nu^N_{k+1} \ra - \la f , \nu^N_k \ra \bigg{)} .\notag
\end{eqnarray}

Therefore, the scaled empirical measure satisfies
\begin{eqnarray}
\la f, \mu^N_{t} \ra - \la f, \mu^N_{0} \ra &=& \sum_{k=0}^{ \floor*{Nt}-1}  \bigg{(} \la f , \nu^N_{k+1} \ra - \la f , \nu^N_k \ra \bigg{)} \notag \\
&=& \sum_{k=0}^{ \floor*{Nt}-1}  \bigg{(} D^{1,N}_k + D^{2,N}_k + \la f,  M^{1,N}(t) \ra + \la f,  M^{2,N}(t) \ra \bigg{)}  +  \frac{1}{N^2} \sum_{k=0}^{ \floor*{N t}-1}  G_k^N \notag \\
&=& \int_0^t   \int_{\mathcal{X}\times\mathcal{Y}}  \alpha \big{(} y -  \la c \sigma(w \cdot x),  \mu^N_{s} \ra \big{)} \la \sigma(w \cdot x) \partial_c f, \mu^N_{s} \ra \pi(dx,dy)   ds\notag \\
& &+ \int_0^t \int_{\mathcal{X}\times\mathcal{Y}}   \alpha \big{(} y -  \la c \sigma(w \cdot x),  \mu^N_{s} \ra \big{)} \la c  \sigma'(w \cdot x) x \cdot \nabla_w f, \mu^N_{s} \ra \pi(dx, dy) ds \notag \\
& &+ \la f, M_t^N \ra+  \frac{1}{N^2} \sum_{k=0}^{ \floor*{N t}-1}  G_k^N + V_t^N
\label{MuEqn}
\end{eqnarray}

Note that $ \frac{1}{N^2} \displaystyle \sum_{k=0}^{ \floor*{N t}-1}  G_k^N$ is $\mathcal{O}(N^{-1})$. Define the fluctuation process
\begin{eqnarray}
\eta_t^N = \sqrt{N} ( \mu_t^N - \bar \mu_t ).\notag
\end{eqnarray}
Then,
\begin{eqnarray}
\la f, \eta_t^N \ra - \la f, \eta^N_0 \ra &=& \int_0^t  \bigg{(} \int_{\mathcal{X}\times\mathcal{Y}}  \alpha \big{(} y -  \la c \sigma(w \cdot x), \bar \mu_{s} \ra \big{)} \la \sigma(w \cdot x) \partial_c f, \eta^N_{s} \ra \pi(dx,dy)  \bigg{)} ds\notag \\
&-&  \int_0^t  \bigg{(} \int_{\mathcal{X}\times\mathcal{Y}}  \alpha \la c \sigma(w \cdot x), \eta_s^N \ra \la \sigma(w \cdot x) \partial_c f, \bar \mu_{s} \ra \pi(dx,dy)  \bigg{)} ds\notag \\
&+& \int_0^t \bigg{(} \int_{\mathcal{X}\times\mathcal{Y}}   \alpha \big{(} y -  \la c \sigma(w \cdot x),  \bar \mu_{s} \ra \big{)} \la c  \sigma'(w \cdot x) x \cdot \nabla_w f, \eta^N_{s} \ra \pi(dx, dy) \bigg{)}ds \notag \\
&-&  \int_0^t \bigg{(} \int_{\mathcal{X}\times\mathcal{Y}}  \alpha\la c \sigma(w \cdot x),  \eta_{s}^N \ra \big{)} \la c  \sigma'(w \cdot x) x \cdot \nabla_w f, \bar \mu_{s} \ra \pi(dx, dy) \bigg{)}ds \notag \\
&+& \sqrt{N} \la f, M_t^N \ra + \Gamma^{1,N}_t +  \Gamma^{2,N}_t  + R_t^{1,N} + R_t^{2,N},
\label{EtaEqn1}
\end{eqnarray}
where
\begin{eqnarray}
 \Gamma^{1,N}_t &=& \frac{1}{\sqrt{N}} \int_0^t  \int_{\mathcal{X}\times\mathcal{Y}} - \alpha   \la c \sigma(w \cdot x), \eta_s^N \ra \la \sigma(w \cdot x) \partial_c f, \eta^N_{s} \ra \pi(dx,dy)   ds \notag \\
\Gamma^{2,N}_t &=& \frac{1}{\sqrt{N}} \int_0^t  \int_{\mathcal{X}\times\mathcal{Y}}  -\alpha   \la c \sigma(w \cdot x), \eta_s^N \ra  \la c \sigma'(w \cdot x) x \nabla_w f, \eta^N_{s} \ra \pi(dx,dy) ds.\notag
\end{eqnarray}
$R_t^{1,N}$ and $R_t^{2,N}$ are $\mathcal{O}(N^{-1/2})$ remainder terms where
\begin{eqnarray}
R_t^{1,N} &=& N^{-3/2} \sum_{k=0}^{  \floor*{N t}-1 } G_k^N, \notag \\
R_t^{2,N} &=& \sqrt{N} V_t^N.\notag
\end{eqnarray}

\section{Relative Compactness} \label{RelativeCompactness}

This section proves the relative compactness of the pre-limit processes $\{\eta_t^N, t\in[0,T]\}_{N\in\mathbb{N}}$ in $D_{W^{-J,2}}([0,T])$ and of  $\{\sqrt{N}M_t^N,t\in[0,T]\}_{N\in\mathbb{N}}$ in $D_{W^{-J,2}}([0,T])$. Lemma \ref{RelativeCompactnessLemma} states that relative compactness of $\{\eta_t^N, t\in[0,T]\}_{N\in\mathbb{N}}$ and of  $\{\sqrt{N}M_t^N,t\in[0,T]\}_{N\in\mathbb{N}}$ in $D_{W^{-J,2}}([0,T])$. The proof is based on Theorem 4.20 of \cite{Kurtz1975}, see also Theorem 8.6 in Chapter 3 of \cite{EthierAndKurtz}. We need to prove that $\eta^{N}_{\cdot}$ and $\sqrt{N}M_{\cdot}$ are appropriately uniformly bounded, see Lemma \ref{CompactcontainmentEtaProcessLemma} and Lemma \ref{L:CompactContainmenrM_process} respectively, and that they satisfy an appropriate regularity type of property, see Lemma \ref{L:regularityEta} and Lemma \ref{L:RegularityMprocess} respectively.

\subsection{Uniform bound on the fluctuations process $\eta^N$}

The main result of this section is Lemma \ref{UniformEtaBoundLemma} below and it provides a uniform bound with respect to $N\in\mathbb{N}$ and $t\in[0,T]$ for the process $\eta^{N}_{t}$.
\begin{lemma} \label{UniformEtaBoundLemma}
If $J_{1} = 2 \ceil*{\frac{D}{2}} + 4$, then there is a constant $C<\infty$ such that
\begin{eqnarray}
\sup_{N \in \mathbb{N}} \sup_{t \in [0,T]} \mathbb{E}  \norm{\eta^N_t}_{-J_{1}}^2  < C.
\label{UnfiormEtaBound}
\end{eqnarray}
\end{lemma}

The proof of this lemma requires a number of intermediate results. We develop these estimates now and present the proof of Lemma \ref{UniformEtaBoundLemma} in the end of this section.

Consider the particle system
\begin{eqnarray}
 \tilde c_t^i &=& c_{0}^i +\int_{0}^{t}\alpha \int_{\mathcal{X}\times\mathcal{Y}}   (y - \la c \sigma (w \cdot x), \bar \mu_s \ra )  \sigma (\tilde w_s^i \cdot x) \pi(dx,dy) ds, \notag \\
\tilde w_t^i &=& w_{0}^i +\int_{0}^{t}\alpha \int_{\mathcal{X}\times\mathcal{Y}}  (y -\la c \sigma (w \cdot x), \bar \mu_s \ra ) \tilde c_s^i \sigma' (\tilde w_s^i \cdot x) x  \pi(dx,dy) ds. \notag \\
\tilde \mu^N_t &=& \frac{1}{N} \sum_{i=1}^N \delta_{ (\tilde c_t^i, \tilde w_t^i )}.\notag
\end{eqnarray}

The particles $(\tilde c^i, \tilde w^i)$ are i.i.d. with law $\bar \mu$ and $\tilde \mu^N \overset{p} \rightarrow \bar \mu$. By the results of  \cite{NeuralNetworkLLN} we obtain that $\tilde \mu^N$ is also compactly supported uniformly in $N\in\mathbb{N}$ and $t\in[0,T]$. We decompose the $\eta_t^N$ into two terms:
\begin{eqnarray}
\eta_t^N = \sqrt{N} ( \mu_t^N - \tilde \mu^N_t ) + \sqrt{N} ( \tilde \mu^N_t - \bar \mu_t ).
\label{DeCompositionEta}
\end{eqnarray}

Define $\Xi^N_t = \sqrt{N} ( \mu_t^N - \tilde \mu^N_t )$. Then,
\begin{eqnarray}
\la f, \Xi^N_t \ra &=& \sqrt{N} \int_0^t  \int_{\mathcal{X}\times\mathcal{Y}}  \alpha \big{(} y -  \la c \sigma(w \cdot x),  \mu^N_{s} \ra \big{)} \la \sigma(w \cdot x) \partial_c f, \mu^N_{s} \ra \pi(dx,dy)   ds\notag \\
& &+ \sqrt{N} \int_0^t  \int_{\mathcal{X}\times\mathcal{Y}}   \alpha \big{(} y -  \la c \sigma(w \cdot x),  \mu^N_{s} \ra \big{)} \la c  \sigma'(w \cdot x) x \cdot \nabla_w f, \mu^N_{s} \ra \pi(dx, dy) ds \notag \\
& &- \sqrt{N} \int_0^t  \int_{\mathcal{X}\times\mathcal{Y}}  \alpha \big{(} y -  \la c \sigma(w \cdot x),  \bar \mu_{s} \ra \big{)} \la \sigma(w \cdot x) \partial_c f, \tilde \mu^N_{s} \ra \pi(dx,dy)   ds\notag \\
& &- \sqrt{N} \int_0^t  \int_{\mathcal{X}\times\mathcal{Y}}   \alpha \big{(} y -  \la c \sigma(w \cdot x), \bar \mu_{s} \ra \big{)} \la c  \sigma'(w \cdot x) x \cdot \nabla_w f, \tilde \mu^N_{s} \ra \pi(dx, dy)  ds \notag \\
& &+ \sqrt{N} \la f, M_t^N \ra+ R_t^{1,N} + R_t^{2,N}.\label{Eq:XiRepresentation}
\end{eqnarray}
By chain rule,
\begin{align}
\la f, \Xi^N_t \ra^2 &= 2 \sqrt{N} \int_0^t   \int_{\mathcal{X}\times\mathcal{Y}}  \alpha \la f, \Xi^N_s \ra \big{(} y -  \la c \sigma(w \cdot x),  \mu^N_{s} \ra \big{)} \la \sigma(w \cdot x) \partial_c f, \mu^N_{s} \ra \pi(dx,dy)   ds\notag \\
&\quad+ 2  \sqrt{N} \int_0^t  \int_{\mathcal{X}\times\mathcal{Y}}   \alpha \la f, \Xi^N_s \ra \big{(} y -  \la c \sigma(w \cdot x),  \mu^N_{s} \ra \big{)} \la c  \sigma'(w \cdot x) x \cdot \nabla_w f, \mu^N_{s} \ra \pi(dx, dy) ds \notag \\
&\quad- 2  \sqrt{N} \int_0^t   \int_{\mathcal{X}\times\mathcal{Y}}  \alpha  \la f, \Xi^N_s \ra \big{(} y -  \la c \sigma(w \cdot x),  \bar \mu_{s} \ra \big{)} \la \sigma(w \cdot x) \partial_c f, \tilde \mu^N_{s} \ra \pi(dx,dy)   ds\notag \\
&\quad- 2 \sqrt{N} \int_0^t  \int_{\mathcal{X}\times\mathcal{Y}}   \alpha  \la f, \Xi^N_s \ra \big{(} y -  \la c \sigma(w \cdot x), \bar \mu_{s} \ra \big{)} \la c  \sigma'(w \cdot x) x \cdot \nabla_w f, \tilde \mu^N_{s} \ra \pi(dx, dy) ds \notag \\
&\quad+  \sum_{k =0}^{\floor*{Nt}-1}  \left( \la f, \Xi_{\frac{k+1}{N}^{-}}^N   +  \sqrt{N} M^{1,N}_k +\sqrt{N} M^{2,N}_k\ra^2   - \la f, \Xi^N_{\frac{k+1}{N}^{-}} \ra^2 \right)  \notag \\
&\quad+ \tilde R_t^{1,N} + \tilde R_t^{2,N}.
\label{Ito}
\end{align}

$\tilde R_t^{1,N}$  and $\tilde R_t^{2,N}$ are the remainder terms
\begin{eqnarray}
\tilde R_t^{1,N} = \sum_{k =0}^{\floor*{Nt}-1} \bigg{(} \big{(} \la f, \Xi_{\frac{k+1}{N}^{-}}^N \ra  + G_k^N  N^{-3/2} \big{)}^2  - \la f, \Xi^N_{\frac{k+1}{N}^{-}} \ra^2  \bigg{)},\notag
\end{eqnarray}
and
\begin{align}
\tilde R_t^{2,N} &= - 2\sqrt{N} \int_{ \frac{ \floor*{Nt} }{N} }^t \int_{\mathcal{X}\times\mathcal{Y}}   \alpha \la f, \Xi_s^N \ra \big{(} y -  \la c \sigma(w \cdot x),  \mu^N_s \ra \big{)} \la \sigma(w \cdot x) \partial_c f, \mu_t^N \ra \pi(dx,dy) ds \notag \\
&- 2\sqrt{N} \int_{ \frac{ \floor*{Nt} }{N} }^t \int_{\mathcal{X}\times\mathcal{Y}}   \alpha  \la f, \Xi_s^N \ra \big{(} y -  \la c \sigma(w \cdot x),  \mu^N_t \ra \big{)} \la c \sigma'(w \cdot x) x \cdot \nabla_w f, \mu_t^N \ra \pi(dx,dy) ds.\nonumber
\end{align}

\begin{lemma}\label{L:RemainderTerms}
With $\tilde R_t^{1,N}$  and $\tilde R_t^{2,N}$ defined as above we have
\begin{align}
|\tilde R_t^{1,N}|+| \tilde R_t^{2,N}  |  &\leq   C_1 \int_{0}^t  \la f, \Xi_s^N \ra^2 ds + C_2 \norm{f}_L^2.\label{Eq:BoundRTerms}
\end{align}
In addition,
\begin{align}
\mathbb{E} \bigg{[}  \sum_{k =0}^{\floor*{Nt}-1} \bigg{(} \big{(} \la f, \Xi^N_{\frac{k+1}{N}^{-}}   + \sqrt{N} M^{1,N}_k + \sqrt{N} M^{2,N}_k \ra\big{)}^2  - \la f, \Xi^N_{ \frac{k+1}{N}^{-}} \ra^2  \bigg{)} \bigg{]}  & \leq C \norm{f}_L^2.
\label{Mbound}
\end{align}
\end{lemma}

The proof of Lemma \ref{L:RemainderTerms} is deferred to Appendix \ref{A:Appendix0}.

 Next, we employ a decomposition into several terms in order to study the first and third term of (\ref{Ito}) (and similarly for the terms two and four of (\ref{Ito})).
\begin{eqnarray}
 &\phantom{.}& \sqrt{N} \bigg{[} \big{(} y -  \la c \sigma(w \cdot x),  \mu^N_{s} \ra \big{)} \la \sigma(w \cdot x) \partial_c f, \mu^N_{s}  \ra  - \big{(} y -  \la c \sigma(w \cdot x),  \bar \mu_{s} \ra \big{)} \la \sigma(w \cdot x) \partial_c f, \tilde \mu^N_{s}  \ra  \bigg{]}  \notag \\
 &=&y \la \sigma(w \cdot x) \partial_c f, \Xi^N_{s}  \ra \notag \\
 & &-   \la c \sigma(w \cdot x),  \mu^N_{s} \ra  \la \sigma(w \cdot x) \partial_c f, \Xi^N_{s} \ra \notag \\
 & &-    \la c \sigma(w \cdot x),  \Xi_s^N \ra  \la \sigma(w \cdot x) \partial_c f, \tilde \mu^N_{s} \ra \notag \\
 & &-   \la c \sigma(w \cdot x), \sqrt{N} ( \tilde \mu_s^N   - \bar \mu_s ) \ra \la \sigma(w \cdot x) \partial_c f, \tilde \mu^N_{s} \ra.\nonumber
\end{eqnarray}

Using the bounds (\ref{Eq:BoundRTerms}) and (\ref{Mbound}) (Lemma \ref{L:RemainderTerms}), equation (\ref{Ito}) gives
\begin{eqnarray}
\mathbb{E} \bigg{[} \la f, \Xi^N_t \ra^2 \bigg{]} &\leq& 2 \mathbb{E} \bigg{[} \int_0^t \int_{\mathcal{X}\times\mathcal{Y}}  \alpha y  \la f, \Xi^N_s  \ra   \la \sigma(w \cdot x) \partial_c f, \Xi^N_{s} \ra \pi(dx,dy)  ds \notag \\
& &+ 2 \int_0^t \int_{\mathcal{X}\times\mathcal{Y}}  \alpha y \la f, \Xi^N_s  \ra   \la c \sigma'(w \cdot x)  x\cdot\nabla_w f, \Xi^N_{s} \ra \pi(dx,dy)  ds \notag \\
& &- 2 \int_0^t \int_{\mathcal{X}\times\mathcal{Y}}  \alpha  \la f, \Xi^N_s  \ra  \la c \sigma(w \cdot x),  \mu^N_{s} \ra  \la \sigma(w \cdot x) \partial_c f, \Xi^N_{s} \ra \pi(dx,dy)  ds \notag \\
& &- 2 \int_0^t   \int_{\mathcal{X}\times\mathcal{Y}}  \alpha  \la f, \Xi^N_s \ra  \la c \sigma(w \cdot x),  \Xi_s^N \ra \la \sigma(w \cdot x) \partial_c f, \tilde \mu^N_{s} \ra \pi(dx,dy)   ds \notag \\
& &-  2 \int_0^t  \int_{\mathcal{X}\times\mathcal{Y}}  \alpha  \la f, \Xi^N_s \ra   \la c \sigma(w \cdot x), \sqrt{N} ( \tilde \mu_s^N   - \bar \mu_s ) \ra  \la \sigma(w \cdot x) \partial_c f, \tilde \mu^N_{s} \ra \pi(dx,dy)  ds \notag \\
& &- 2  \int_0^t  \int_{\mathcal{X}\times\mathcal{Y}}  \alpha  \la f, \Xi^N_s \ra   \la c \sigma(w \cdot x),  \mu^N_{s} \ra  \la  c \sigma'(w \cdot x) x \cdot \nabla_w f, \Xi^N_{s} \ra \pi(dx,dy)  ds \notag \\
& &- 2 \int_0^t   \int_{\mathcal{X}\times\mathcal{Y}}  \alpha  \la f, \Xi^N_s \ra   \la c \sigma(w \cdot x),  \Xi_s^N \ra  \la c \sigma'(w \cdot x) x \cdot \nabla_w f, \tilde \mu^N_{s} \ra \pi(dx,dy)   ds \notag \\
& &- 2 \int_0^t \int_{\mathcal{X}\times\mathcal{Y}}  \alpha  \la f, \Xi^N_s \ra  \la c \sigma(w \cdot x), \sqrt{N} ( \tilde \mu_s^N   - \bar \mu_s ) \ra  \la c  \sigma'(w \cdot x)  x \cdot \nabla_w f, \tilde \mu^N_{s} \ra \pi(dx,dy)   ds \bigg{]} \notag \\
& &+   C_1 \int_{0}^t  \mathbb{E} \left[\la f, \Xi_s^N \ra^2 \right]ds + C_2 \norm{f}_L^2.
\label{XiSquared}
\end{eqnarray}

We begin with the fourth term in (\ref{XiSquared}); the seventh term can be treated completely analogously and is omitted.  First, notice that for any $x \in \mathcal{X}$,
\begin{eqnarray}
 \la c \sigma(w \cdot x),  \Xi_s^N \ra^2  &\leq&  \norm{  c \sigma(w \cdot x)}_{J_{1}}^2  \norm{\Xi_s^N}_{-J_{1}}^2\leq  C \norm{\Xi_s^N}_{-J_{1}}^2,
 \label{NewBoundXicsigma}
\end{eqnarray}
due to Assumption \ref{A:Assumption1} and to the compactness of $\Theta$.

By the Sobolev embedding Theorem (Theorem 6.2 in \cite{Adams}), we have that
\begin{eqnarray}
 \sum_{|\alpha| \leq 2 } \sup_{c,w \in K} |D^{\alpha} f(c,w)  |  \leq   C\norm{ f}_{L}
\label{SupToJnorm}
\end{eqnarray}
where $L = \ceil*{\frac{D}{2}} + 3$.

Using Young's inequality, (\ref{NewBoundXicsigma}), and (\ref{SupToJnorm}) to bound $\la \sigma(w \cdot x) \partial_c f, \tilde \mu^N_{s} \ra^2 \leq C  \big{(} \sum_{| \alpha | = 1} \sup_{(c,w) \in K} |D^{\alpha}f( c,w) | \big{)}^2 \leq C \norm{ f}_L^2$, we obtain
\begin{eqnarray}
&\phantom{.}&- \int_0^t  \int_{\mathcal{X}\times\mathcal{Y}}  \alpha  \la f, \Xi^N_s \ra  \la c \sigma(w \cdot x),  \Xi_s^N \ra \la \sigma(w \cdot x) \partial_c f, \tilde \mu^N_{s} \ra \pi(dx,dy)  ds \notag \\
&\leq& C  \int_0^t  \int_{\mathcal{X}\times\mathcal{Y}} \bigg{(}  \la f, \Xi^N_s \ra^2 +  \norm{\Xi_s^N}_{-J_{1}}^2  \norm{f}_L^2 \bigg{)} \pi(dx,dy)  ds.\notag
\end{eqnarray}

Next, we study the fifth term in (\ref{XiSquared}); the eighth term can be treated completely analogously and is omitted. The term $ \la c \sigma(w \cdot x), \sqrt{N} ( \tilde \mu_s^N   - \bar \mu_s ) \ra$ can be re-written as
\begin{eqnarray}
 \la c \sigma(w \cdot x), \sqrt{N} ( \tilde \mu_t^N   - \bar \mu_t ) \ra = N^{-1/2} \sum_{i=1}^N  \big{(} \tilde c_t^i  \sigma (  \tilde w_t^i  x ) -  \la c \sigma( w x), \bar \mu_t \ra   \big{)}.\notag
\end{eqnarray}
Since $(\tilde c_t^i, \tilde w_t^i )$ are i.i.d. random variables with law $\bar \mu_t$ and $x$ takes values in the compact set $\mathcal{X}$,
\begin{eqnarray}
\mathbb{E} \bigg{[}  \la c \sigma(w \cdot x), \sqrt{N} ( \tilde \mu_t^N   - \bar \mu_t ) \ra^{2} \bigg{]} \leq  C.\notag
\end{eqnarray}

Using Young's inequality and the fact that $\bar \mu$ takes values in a compact set $K$,
\begin{eqnarray}
&\phantom{.}& \mathbb{E} \bigg{[}  \int_0^t  \int_{\mathcal{X}\times\mathcal{Y}}  -  \la f, \Xi^N_s \ra   \la c \sigma(w \cdot x), \sqrt{N} ( \tilde \mu_s^N   - \bar \mu_s ) \ra \la \sigma(w \cdot x) \partial_c f, \tilde \mu^N_{s} \ra \pi(dx,dy)  ds \bigg{]} \notag \\
&\leq& C \int_0^t  \int_{\mathcal{X}\times\mathcal{Y}}    \mathbb{E} \big{[} \la f, \Xi^N_s \ra^2  \big{]}+  \mathbb{E} \big{[} \la c \sigma(w \cdot x), \sqrt{N} ( \tilde \mu_s^N   - \bar \mu_s ) \ra ^2 \la \sigma(w \cdot x) \partial_c f, \tilde \mu^N_{s} \ra^2 \big{]} \pi(dx,dy)  ds \notag \\
&\leq& C \int_0^t  \int_{\mathcal{X}\times\mathcal{Y}}    \mathbb{E} \big{[} \la f, \Xi^N_s \ra^2  \big{]}+  \mathbb{E} \big{[}   \la c \sigma(w \cdot x), \sqrt{N} ( \tilde \mu_s^N   - \bar \mu_s ) \ra ^2 \big{]} ( \sup_{c,w \in K} | \partial_{c}f | )^2  \pi(dx,dy)  ds \notag \\
&\leq& C \int_0^t  \int_{\mathcal{X}\times\mathcal{Y}}  \bigg{(}  \mathbb{E} \big{[} \la f, \Xi^N_s \ra^2  \big{]}+ \norm{f}_L^2 \bigg{)} \pi(dx,dy)  ds.\notag
\end{eqnarray}


Hence, it remains to study the first, second, third and sixth term in (\ref{XiSquared}). To do so, we first state the following lemma.
\begin{lemma} \label{CompactLemmaXi}
There is a compact set $K = [-C_{o}, C_{o}]^{1+d} \subset \mathbb{R}^{1+d}$ such that $\eta_t^N$, $\Xi_t^N$, and $\mu_t^N$ vanish when evaluated on any $A\subset K^{c}$.
\end{lemma}
\begin{proof}
Due to the uniform bound (\ref{UniformBoundfromLLNpaper}), there exists a compact set $K \subset \mathbb{R}^{1+d}$ such that $\mu_t^N(K^c) = \tilde \mu_t^N(K^c) = \bar \mu_t(K^c) = 0$. It directly follows from the definitions of $\eta_t^N$ and $\Xi_t^N$ that they also vanish outside of the set $K$. For example, for $A \in K^c$, $\eta_t^N(A) = \sqrt{N} \big{(} \mu_t^N ( A) - \bar \mu_t(A) \big{)} = 0$.
\end{proof}

Due to Lemma \ref{CompactLemmaXi}, there is a $C^{\infty}_c$ ``bump" function $b(c,w)$  such that $ b(c,w) c\sigma'(w x)$ is in $C^{\infty}_c( \mathbb{R}^{1+d} \times \mathcal{X})$ and $ b(c,w) c\sigma'(w x) = c \sigma'(wx)$ for every $(c,w) \in K$, the compact set defined in Lemma \ref{CompactLemmaXi}, and $x\in \mathcal{X}$. Similar statements hold for the terms $ \sigma (w x)$ and $c \sigma (w x)$.  See \cite{Fry} for a discussion on bump functions. An example of a bump function is:
\begin{eqnarray}
b(z) &=& \frac{ h \big{(} 2- \frac{ \norm{z} }{ r} \big{)} }{ h \big{(} \frac{\norm{z}}{r} -1 \big{)} + h \big{(} 2 -  \frac{ \norm{z}}{r} \big{)} }, \notag \\
h(v) &=& e^{ - \frac{1}{v^2} } \mathbf{1}_{ v > 0}.
\label{BumpFunctionDefinition}
\end{eqnarray}
The function $b(z)$ is $C^{\infty}_c(\mathbb{R}^{1+d})$, vanishes for $\norm{z} \geq 2 r$, and is one on $\norm{z} \leq r$ \cite{Fry}. For the purposes of this paper, we may choose $r = \sqrt{D} C_{o}$, $B = 3 \sqrt{D} C_{o}$, and $\Theta = (-B, B)^D$. In particular, notice for instance that $b(c,w) c\sigma'(w x)$, and its partial derivatives, vanish on the boundary of $\Theta$.

Going back to (\ref{XiSquared}), the aforementioned discussion implies that we can write for example
\begin{align}
 \la c \sigma'(w \cdot x)  x\cdot \nabla_w f, \Xi^N_{s} \ra&= \la b(c,w) c \sigma'(w \cdot x)  x\cdot \nabla_w f, \Xi^N_{s} \ra\nonumber\\
  \la \sigma(w \cdot x) \partial_c f, \Xi^N_{s} \ra&= \la b(c,w)\sigma(w \cdot x) \partial_c f, \Xi^N_{s} \ra\nonumber
\end{align}



Hence, let us define the operators
\begin{eqnarray}
\mathcal{G}_1 f &=&  b(c,w) c \sigma'(w x) x \cdot \nabla_{w}f, \notag \\
\mathcal{G}_2 f &=&  b(c,w) \sigma(w x)  \partial_{c}f. \notag \\
\label{operatorsG}
\end{eqnarray}

Let $\{ f_a \}_{a=1}^{\infty}$ be a complete orthonormal basis for $W_0^{J_{1},2}(\Theta)$. Since $J_{1} -L > \frac{D}{2}$, the embedding $ W_0^{J_{1},2}(\Theta) \hookrightarrow W_0^{L,2}(\Theta)$ is of Hilbert-Schmidt type and
\begin{eqnarray}
\sum_{a} \norm{ f_a }_L^2 < \infty.
\label{HSbound}
\end{eqnarray}
(See Theorem 6.53 of \cite{Adams} for details.)

Let $f = f_a$ in (\ref{XiSquared}) and sum over all $a \geq 1$. Using Parseval's identity, we now have the bound
\begin{eqnarray}
\mathbb{E} \bigg{[}  \norm{\Xi^N_t }_{-J_{1}}^2 \bigg{]}  &\leq& C_1  \int_0^t  \int_{\mathcal{X}\times\mathcal{Y}} \mathbb{E} \bigg{[} \norm{\Xi^N_s}_{-J_{1}}^2 + \left|\big{(} y -  \la c \sigma(w \cdot x),  \mu^N_{s} \ra \big{)}\right| \left|\la \Xi^N_s, \mathcal{G}_1^{\ast} \Xi^N_s \ra_{-J_{1}} \right| \notag \\
&+& \left|\big{(} y -  \la c \sigma(w \cdot x),  \mu^N_{s} \ra \big{)}\right|\left|  \la \Xi^N_s, \mathcal{G}_2^{\ast} \Xi^N_s \ra_{-J_{1}}   \right|  \bigg{]} \pi(dx,dy)   ds +  C_2 .
\label{XiSquared2}
\end{eqnarray}

Since $\mu_t^N$ takes values in a compact set and $\mathcal{X} \times \mathcal{Y}$ is compact, we have that:
\begin{eqnarray}
\mathbb{E} \bigg{[}  \norm{\Xi^N_t }_{-J_{1}}^2 \bigg{]} &\leq& C_1  \int_0^t  \int_{\mathcal{X}\times\mathcal{Y}} \mathbb{E} \bigg{[} \norm{\Xi^N_s}_{-J_{1}}^2 + | \la \Xi^N_s, \mathcal{G}_1^{\ast} \Xi^N_s \ra_{-J_{1}} |  \notag \\
&+& | \la \Xi^N_s, \mathcal{G}_2^{\ast} \Xi^N_s \ra_{-J_{1}}|     \bigg{]} \pi(dx,dy)   ds +  C_2,
\label{XiSquared3}
\end{eqnarray}
for some unimportant but finite constants $C_1,C_2<\infty$.

The terms $| \la \Xi^N_s, \mathcal{G}_1^{\ast} \Xi^N_s \ra_{-J_{1}} |$  and $| \la \Xi^N_s, \mathcal{G}_2^{\ast} \Xi^N_s \ra_{-J_{1}}|$ must now be analyzed. By the Riesz representation theorem for Hilbert spaces, for $\Xi \in W^{-J_{1},2}$  there exists a unique $\Psi = F(\Xi) \in W_0^{J_{1},2}$ such that,
\begin{eqnarray*}
\la f, \Xi \ra &=& \la f, \Psi \ra_{J_{1}}, \text{ for }f \in W_0^{J_{1},2}.
\end{eqnarray*}


\begin{lemma} \label{LemmaBoundPsiG}
For $\Xi \in W^{-J,2}$ with $J\geq J_{1}=2\ceil*{\frac{D}{2}}+4$, we have
\begin{eqnarray}
\left|\la \Xi, \mathcal{G}_1^{\ast} \Xi \ra_{-J}\right| &\leq& C \norm{\Xi}_{-J}^2, \notag \\
\left|\la \Xi, \mathcal{G}_2^{\ast} \Xi \ra_{-J}\right| &\leq& C \norm{\Xi}_{-J}^2.\notag
\end{eqnarray}
\end{lemma}
\begin{proof}
 Notice that $\{\Xi\in W^{-J,2}: F(\Xi)\in W^{J+1,2}_{0}\}$ is dense in $W^{-J,2}$. For $\Xi \in W^{-J,2}$ such that $\Psi = F(\Xi) \in W_0^{J+1,2}$ we have by definition
\begin{eqnarray}
\la \Xi, \mathcal{G}^{\ast} \Xi \ra_{-J} = \la \Psi, \mathcal{G}^{\ast} \Xi \ra = \la \mathcal{G} \Psi, \Xi \ra  = \la \mathcal{G} \Psi, \Psi \ra_{J},\notag
\end{eqnarray}
since $ \mathcal{G}\Psi\in W_0^{J,2}$ for either $ \mathcal{G}= \mathcal{G}_{1}$ or $ \mathcal{G}= \mathcal{G}_{2}$. 
%
%


By setting $g(c,w,x) = b(c,w) c \sigma'(w x) x$ in Lemma \ref{LemmaBoundPsi} and  $\mathcal{X}$ being a compact set,
\begin{eqnarray}
\left|\la \mathcal{G}_1 \Psi, \Psi \ra_{J}\right|  &\leq& C \norm{\Psi }_{J}^2 = C \norm{ \Xi }_{-J}^2. \notag
\end{eqnarray}
Note that we have also used the fact that uniformly in $c,w,x$
\begin{eqnarray*}
\bigg{|} \frac{\partial^k g}{\partial c^{k_1} \partial k_2}(c,w, x) \bigg{|} \leq C,
\end{eqnarray*}
for any $0 \leq k \leq J$, $k = k_1 + k_2$, and $k_1, k_2 \geq 0$. This is due to $\sigma(\cdot) \in C_b^{\infty}$, $b(c,w) \in C_c^{\infty}$, and $\mathcal{X}$ being a compact set.


Similarly, setting $g(c,w,x) =  b(c,w) \sigma(w x) $ in Lemma  \ref{LemmaBoundPsi} and since $\mathcal{X}$ is a compact set,
\begin{eqnarray}
\left|\la \mathcal{G}_2 \Psi, \Psi \ra_{J}\right|  &\leq& C \norm{\Psi }_{J}^2 = C \norm{ \Xi }_{-J}^2. \notag
\end{eqnarray}

In the previous two bounds, $C<\infty$ is a finite constant that depends on $\mathcal{X}$ as well as on $C_{o}$ from Lemma \ref{CompactLemmaXi}. The proof of the lemma is now complete.
\end{proof}

Lemma \ref{LemmaBoundPsiG} and equation (\ref{XiSquared3}) produce the bound
\begin{eqnarray}
\mathbb{E}  \norm{\Xi^N_t }_{-J_{1}}  &\leq& C_1  \int_0^t  \mathbb{E}  \norm{\Xi^N_s}_{-J_{1}}^2     ds +  C_2 .\notag
\end{eqnarray}
Note that we have again used the fact that $\mathcal{X} \times \mathcal{Y}$ is a compact set.

By Gronwall's Lemma,
\begin{eqnarray}
\sup_{N \in \mathbb{N}} \sup_{t \in [0,T]} \mathbb{E}  \norm{\Xi^N_t}_{-J_{1}}^2  < C.
\label{XiFinalBound}
\end{eqnarray}

Recall that
\begin{eqnarray}
\la f, \eta_t^N \ra = \la f,  \Xi_t^N \ra + \la f, \sqrt{N} ( \tilde \mu_t^N - \bar \mu_t ) \ra.\notag
\end{eqnarray}
Therefore,
\begin{eqnarray}
\la f, \eta_t^N \ra^2 \leq 2 \la f,  \Xi_t^N \ra^2 + 2 \la f, \sqrt{N} ( \tilde \mu_t^N - \bar \mu_t ) \ra^2.
\label{DecompositionFinalEta}
\end{eqnarray}

The second term is  a sequence of i.i.d. random variables. That is,
\begin{eqnarray}
\mathbb{E} \la f, \sqrt{N} ( \tilde \mu_t^N - \bar \mu_t ) \ra^2 &=& \mathbb{E} \bigg{[} \bigg{(} \frac{1}{\sqrt{N}} \sum_{i=1}^N [ f( \tilde c_t^i, \tilde w_t^i) - \la f, \bar \mu_t \ra ] \bigg{)}^2 \bigg{]} \notag \\
&\leq& \mathbb{E} \bigg{[}  \frac{1}{N} \sum_{i=1}^N \sum_{j=1}^N [ f( \tilde c_t^i, \tilde w_t^i) - \la f, \bar \mu_t \ra ] [ f( \tilde c_t^j, \tilde w_t^j) - \la f, \bar \mu_t \ra ] \bigg{]} \notag \\
&= &  \big{(} \la f^2, \bar \mu_t \ra - \la f, \bar \mu_t \ra^2 \big{)} \notag \\
&\leq& C \norm{ f}_L^2,
\label{iidCLTbound}
\end{eqnarray}
where the last inequality follows from the compact support of $\bar \mu$ and the bound (\ref{SupToJnorm}).

Now, we are in position to complete the proof of Lemma \ref{UniformEtaBoundLemma}.
\begin{proof}[Proof of Lemma \ref{UniformEtaBoundLemma}]
Let $f = f_a$ in (\ref{DecompositionFinalEta}) and sum over all $a \geq 1$. The lemma follows from Parseval's identity, (\ref{HSbound}), (\ref{XiFinalBound}), and (\ref{iidCLTbound}).
\end{proof}

\subsection{Compact containment of $\sqrt{N} M^N$}\label{SS:CompactContainmentM}
Let $\mathfrak{F}_{t}$ be the $\sigma-$algebra generated by $\mu^{N}_{s}$ and $M^{N}_{s}$ for $s\leq t$ and note that $\la f, \sqrt{N} M_t^N \ra$ is a $\mathfrak{F}_{t}-$martingale. Indeed, as in Lemma 3.1 of \cite{NeuralNetworkLLN} we get 
\begin{eqnarray}
\mathbb{E} \bigg{[} \la f, \sqrt{N} M_t^N \ra \bigg{|} \mathfrak{F}_{r} \bigg{]} &=& \mathbb{E} \bigg{[} \la f, \sqrt{N} M_t^N  - \sqrt{N} M_r^N \ra \bigg{|}  \mathfrak{F}_{r} \bigg{]}
+ \mathbb{E} \bigg{[} \la f, \sqrt{N} M_r^N \ra \bigg{|} \mathfrak{F}_{r} \bigg{]} \notag \\
&=& \sum_{ k = \floor*{N r}}^{ \floor*{N t} -1 }   \mathbb{E}    \bigg{[} \bigg{(}  \la f, M^{1,N}_k \ra + \la f, M^{2,N}_k \ra \bigg{|} \mathcal{F}_{\floor*{Nr}}^{N} \bigg{]}  + \la f, \sqrt{N} M_r^N \ra \notag \\
&=&  \la f, \sqrt{N} M_r^N \ra.\notag
\end{eqnarray}

and in addition, for every $t\leq T$, the quadratic variation of $\la f, \sqrt{N} M_t^N \ra$ is seen to satisfy
\begin{equation*}
\mathbb{E}\left[\la f, \sqrt{N} M_{\cdot}^N \ra\right]_{t}\leq C  \norm{f}_L^2<\infty,
\end{equation*}

The last inequality is proven using the same approach as in equation (\ref{Mbound}) and using  Lemma 3.1 of \cite{NeuralNetworkLLN} (see also the derivation of (\ref{QuadraticVariationMpreLimit}) later on). Let us also recall that $L = \ceil*{\frac{D}{2}} + 3$.

Then,  Doob's martingale inequality yields
\begin{eqnarray}
\mathbb{E} \bigg{[} \sup_{t \in [0,T]}   \la f, \sqrt{N} M_t^N \ra^2  \bigg{]} \leq C \mathbb{E} \bigg{[} \la f,  \sqrt{N} M_T^N \ra^2 \bigg{]}\leq C  \norm{f}_L^2.
\label{BoundM001100}
\end{eqnarray}

\begin{lemma}\label{L:CompactContainmenrM_process}
If $J_{1} = 2 \ceil*{\frac{D}{2}} + 4$, then there is a constant $C<\infty$ such that
\begin{eqnarray}
\sup_{N \in \mathbb{N}} \mathbb{E} \bigg{[}  \sup_{t \in [0,T]}  \norm{\sqrt{N} M^N_t}_{-J_{1}}^2 \bigg{]} \leq C.
\label{CompactContainmentMEqn}
\end{eqnarray}
\end{lemma}
\begin{proof}

Let $\{f_a\}_{a\geq 1}$ be a complete orthonormal basis for $W^{J_{1},2}_{0}(\Theta)$. Using equation (\ref{BoundM001100}), Parseval's identity and (\ref{HSbound}) we get
\begin{align}
\sup_{N \in \mathbb{N}} \mathbb{E} \bigg{[}  \sup_{t \in [0,T]}  \norm{\sqrt{N} M^N_t}_{-J_{1}}^2 \bigg{]} &= \sup_{N \in \mathbb{N}} \mathbb{E} \bigg{[}  \sup_{t \in [0,T]}  \sum_{a \geq 1}  \la f_a, \sqrt{N} M_t^N \ra^2   \bigg{]}  \leq  \sup_{N \in \mathbb{N}} \sum_{ a \geq 1} \norm{f_a}_L^2 < \infty ,\notag
\end{align}
completing the proof of the lemma.
\end{proof}


\subsection{Regularity of $\sqrt{N} M^N$}
Let $\{ f_a \}_{a=1}^{\infty}$ be a complete orthonormal basis for $W_0^{J_{1},2}$  with $J_{1} = 2  \ceil*{\frac{D}{2}} + 4$. For $0 \leq r < t < T$. The equation below is the sum of jump terms at discrete times $\frac{1}{N}, \frac{2}{N}, \ldots, \frac{T}{N}$. By Lemma 3.1 of \cite{NeuralNetworkLLN}  (see also Theorem 3.2 of \cite{Burkholder}) and (\ref{SupToJnorm}) we get that for $0\leq r<t\leq T$ with  $(t - r) < \delta<1$
\begin{eqnarray}
&\phantom{.}& \mathbb{E} \bigg{[} \bigg{(} \la f_{a}, \sqrt{N} M_t^N \ra -  \la f_{a},  \sqrt{N} M_r^N  \ra \bigg{)}^2  \bigg{|} \mathfrak{F}_{r} \bigg{]} \notag \\
&=& N \mathbb{E} \bigg{[}\sum_{ k = \floor*{N r}}^{ \floor*{N t} -1 }      \bigg{(}  \la f_{a}, M^{1,N}_k \ra + \la f_{a}, M^{2,N}_k \ra \bigg{)}^2 \bigg{|} \mathcal{F}_{\floor*{Nr}}^{N}\bigg{]} \notag \\
&\leq& \frac{C}{N} \mathbb{E} \bigg{[}\sum_{ k = \floor*{N r}  }^{ \floor*{N t}-1 }   \sum_{|\alpha|= 1 } \sup_{c,w \in K} |D^{\alpha} f(c,w)  | \bigg{]} \notag \\
&\leq& C_1 \norm{f_{a}}_L^2 \delta  + \frac{C_2}{N} \norm{f_{a}}_L^2,\label{RegularityMf11}
\end{eqnarray}

where $L = \ceil*{\frac{D}{2}} + 3$ and the third line is derived using the same approach as in equation (\ref{Mbound}).
 By Parseval's identity and (\ref{HSbound}), we get Lemma \ref{L:RegularityMprocess} below.

\begin{lemma}\label{L:RegularityMprocess}
Let $J_{1} = 2  \ceil*{\frac{D}{2}} + 4$. If $0\leq r<t\leq T$ are such that  $(t - r) < \delta<1$, then there are unimportant constants $C_1,C_2<\infty$
\begin{eqnarray}
\mathbb{E} \bigg{[}  \norm{ \sqrt{N} M_t^N  -  \sqrt{N} M_r^N  }_{-J_{1}}^2   \bigg{|} \mathfrak{F}_{r} \bigg{]} \leq C_1  \delta  + \frac{C_2}{N}.\label{RegularityM}
\end{eqnarray}
\end{lemma}

In particular, (\ref{RegularityM}) implies that the regularity condition of Theorem 4.20 in \cite{Kurtz1975} (equivalently  condition B of Theorem 8.6 of Chapter 3 of \cite{EthierAndKurtz}) is satisfied. (See also Remark 8.7 B of Chapter of \cite{EthierAndKurtz} regarding replacing $\sup_N$ with $\lim_N$.) 


\subsection{Regularity of $\eta^N$}

Let $\{ f_a \}_{a=1}^{\infty}$ be a complete orthonormal basis for $W_0^{J_{2},2}(\Theta)$ for $J_{2} = 3  \ceil*{\frac{D}{2}} + 6$. Recall that $\eta_t^N$ can be written via the decomposition
\begin{eqnarray}
\eta_t^N = \sqrt{N} ( \mu_t^N - \tilde \mu^N_t ) + \sqrt{N} ( \tilde \mu^N_t - \bar \mu_t ).
\label{DeCompositionEtaRepeat}
\end{eqnarray}
Let $\Xi_t^N = \sqrt{N} ( \mu_t^N - \tilde \mu^N_t )$ and $Z_t^N = \sqrt{N} ( \tilde \mu^N_t - \bar \mu_t )$. For $0 \leq r < t < T$,
\begin{eqnarray}
\big{(} \la f_a, \eta_t^N \ra - \la f_a, \eta_r^N \ra \big{)}^2 \leq 2 \big{(} \la f_a, \Xi_t^N \ra -  \la f_a, \Xi_r^N \ra  \big{)}^2 +  2 \big{(} \la f_a, Z_t^N \ra - \la f_a, Z_r^N \ra \big{)}^2.\notag
\end{eqnarray}

Using Young's inequality, the Cauchy-Schwarz inequality, and the fact that $[0,T] \times \mathcal{X} \times \mathcal{Y}$ is a compact set,
\begin{align}
\big{(} \la f_{a}, \Xi_t^N \ra -  \la f_{a}, \Xi_r^N \ra  \big{)}^2 &\leq C \bigg{[} \int_r^t \int_{\mathcal{X}\times\mathcal{Y}}   \bigg{(}  y  \la \sigma(w \cdot x) \partial_c f_{a}, \Xi^N_{s} \ra \bigg{)}^2 \pi(dx,dy)  ds \notag \\
&\quad+ \int_r^t \int_{\mathcal{X}\times\mathcal{Y}}  \bigg{(} y \la c \sigma'(w \cdot x)  x\cdot\nabla_w f_{a}, \Xi^N_{s} \ra \bigg{)}^2 \pi(dx,dy)  ds \notag \\
&\quad+  \int_r^t \int_{\mathcal{X}\times\mathcal{Y}} \bigg{(}    \la c \sigma(w \cdot x),  \mu^N_{s} \ra  \la \sigma(w \cdot x) \partial_c f_{a}, \Xi^N_{s} \ra  \bigg{)}^2 \pi(dx,dy)  ds \notag \\
&\quad+  \int_r^t   \int_{\mathcal{X}\times\mathcal{Y}}  \bigg{(}  \la c \sigma(w \cdot x),  \Xi_s^N \ra \la \sigma(w \cdot x) \partial_c f_{a}, \tilde \mu^N_{s} \ra \bigg{)}^2 \pi(dx,dy)   ds \notag \\
&\quad+   \int_r^t  \int_{\mathcal{X}\times\mathcal{Y}}   \bigg{(}   \la c \sigma(w \cdot x), \sqrt{N} ( \tilde \mu_s^N   - \bar \mu_s ) \ra  \la \sigma(w \cdot x) \partial_c f_{a}, \tilde \mu^N_{s} \ra  \bigg{)}^2 \pi(dx,dy)  ds \notag \\
&\quad+   \int_r^t  \int_{\mathcal{X}\times\mathcal{Y}}   \bigg{(}    \la c \sigma(w \cdot x),  \mu^N_{s} \ra  \la  c \sigma'(w \cdot x) x \cdot \nabla_w f_{a}, \Xi^N_{s} \ra  \bigg{)}^2 \pi(dx,dy)  ds \notag \\
&\quad+  \int_r^t   \int_{\mathcal{X}\times\mathcal{Y}}  \bigg{(}    \la c \sigma(w \cdot x),  \Xi_s^N \ra  \la c \sigma'(w \cdot x) x \cdot \nabla_w f_{a}, \tilde \mu^N_{s} \ra \bigg{)}^2 \pi(dx,dy)   ds \notag \\
&\quad+  \int_r^t \int_{\mathcal{X}\times\mathcal{Y}}   \bigg{(}  \la c \sigma(w \cdot x), \sqrt{N} ( \tilde \mu_s^N   - \bar \mu_s ) \ra  \la c  \sigma'(w \cdot x)  x \cdot \nabla_w f_{a}, \tilde \mu^N_{s} \ra  \bigg{)}^2 \pi(dx,dy)   ds \notag \\
&\quad+  \big{(} \la f_{a}, \sqrt{N}M_t^N \ra - \la f_{a}, \sqrt{N}M_r^N \ra \big{)}^2    + ( R_t^{1,N} - R_r^{1,N} )^2  + ( R_t^{2,N} - R_r^{2,N} )^2 \bigg{]}.
\label{XiRegularity11}
\end{align}

Recall that $J_1 = 2 \ceil*{\frac{D}{2}} + 4$. Since $\pi(dx,dy)$ has compact support, we have 
\begin{eqnarray}
\bigg{|} \la  c \sigma'(w \cdot x) x \cdot \nabla_w f_{a}, \Xi^N_{s} \ra \bigg{|}  
&\leq& \norm{ c \sigma'(w \cdot x) x \cdot \nabla_w f_{a}} _{J_1} \norm{ \Xi^N_s }_{ - J_1 } \notag \\
&\leq&  C \norm{ f_{a} }_{J_1 + 1 }  \norm{ \Xi^N_s }_{ - J_1 },
\label{J1bound}
\end{eqnarray}
and, we have by (\ref{XiFinalBound}) that $\sup_{N \in \mathbb{N}} \sup_{t \in [0,T]} \mathbb{E}  \norm{\Xi^N_t}_{-J_{1}}^2  < C$. This treats the second term on the right hand side of (\ref{XiRegularity11}). The treatment of the first, third, fourth, sixth and seventh term is basically almost identical. The fifth and eighth terms are handled using the (\ref{iidCLTbound}), (\ref{SupToJnorm}) the fact that $\mathcal{X} \times \mathcal{Y}$ is a compact set, and the compact containment of $\mu_t^N$ and $\tilde \mu_t^N$.

Recall also that $R_t^{1,N} = N^{-3/2} \displaystyle \sum_{k=0}^{  \floor*{N t} -1 } G_k^N$ where $|G_k^N|  \leq \norm{f}_L$ where $L = \ceil*{\frac{D}{2}} + 3$. Therefore,
\begin{align}
(R_t^{1,N} - R_r^{1,N} )^2 &= \bigg{(} N^{-3/2} \sum_{k= \floor*{Nr} }^{  \floor*{N t}-1 } G_k^N \bigg{)}^2 \leq \frac{1}{N}  (t - r)^2  \norm{f_{a}}_L^2 + \frac{1}{N^3} \norm{f_{a}}_L^2.
\label{Rregularity11}
\end{align}

In addition, by (\ref{Vbound}) and (\ref{SupToJnorm}) 
\begin{align}
( R_t^{2,N} - R_r^{2,N} )^2 &\leq 2 \big{(} \sqrt{N} V_t^N \big{)}^2 + 2 \big{(} \sqrt{N} V_r^N \big{)}^2 \leq \frac{C}{N} \norm{ f_{a} }_L^2.\label{Rregularity12}
\end{align}

Therefore, using (\ref{XiRegularity11}), (\ref{J1bound}), (\ref{Rregularity11}), (\ref{Rregularity12}), (\ref{iidCLTbound}), (\ref{SupToJnorm}) the fact that $\mathcal{X} \times \mathcal{Y}$ is a compact set, the compact containment of $\mu_t^N$ and $\tilde \mu_t^N$, and (\ref{RegularityMf11}) for the martingale terms, we have for $0\leq r<t\leq T$ with $(t-r)<\delta<1$
\begin{align}
\mathbb{E} \bigg{[} \big{(} \la f_{a}, \Xi_t^N \ra -  \la f_{a}, \Xi_r^N \ra  \big{)}^2 \bigg{|} \mathfrak{F}_{r} \bigg{]} &\leq C  \bigg{[} \sup_{0\leq s\leq T}\mathbb{E}  \norm{\Xi_s^N}_{-J_1}^2  \bigg{(} \norm{f_{a}}_L^2   + \norm{f_{a}}_{J_1+1}^2   \bigg{)}  \delta \notag \\
&\quad +  \norm{f_{a}}_L^2  \delta  +  \frac{1}{N}   \norm{f_{a}}_L^2 \delta^2 \bigg{]} + C_2 \big{(} \frac{1}{N} + \frac{1}{N^3}  \big{)} \norm{f_{a}}_L^2  .
\label{XiRegularity22}
\end{align}

Since we have chosen $J_{2}=3\ceil*{\frac{D}{2}}+6$, we certainly have that $J_{2}>J_{1}+1+\frac{D}{2}$, which then implies that
$\sum_{a\geq 1}\norm{f_{a}}_{J_1+1}^2<\infty$ and $\sum_{a\geq 1}\norm{f_{a}}_{L}^2<\infty$. Hence, using the uniform bound (\ref{XiFinalBound}) and Parseval's identity, we obtain for $(t-r)<\delta<1$ that
\begin{eqnarray}
\mathbb{E} \bigg{[} \norm{ \Xi_t^N  -   \Xi_r^N }_{-J_{2}}^2 \bigg{|} \mathfrak{F}_{r} \bigg{]} \leq C_1 \delta + C_2  \frac{1}{N}.\notag
\end{eqnarray}
Using a similar approach, we can show that  for $(t-r)<\delta<1$ (see also (\ref{iidCLTbound})) there is a finite constant $C<\infty$ such that
\begin{eqnarray}
\mathbb{E} \bigg{[} \norm{ Z_t^N  -   Z_r^N }_{-J_{2}}^2 \bigg{|} \mathfrak{F}_{r} \bigg{]} \leq C \delta.\notag
\end{eqnarray}

\begin{lemma}\label{L:regularityEta}
Let $J_{2} = 3  \ceil*{\frac{D}{2}} + 6$. If $0\leq r<t\leq T$ are such that  $(t - r) < \delta<1$, then there are unimportant constants $C_1,C_2<\infty$ such that
\begin{eqnarray}
\mathbb{E} \bigg{[} \norm{ \eta_t^N  -   \eta_r^N }_{-J_{2}}^2 \bigg{|} \mathfrak{F}_{r}  \bigg{]} \leq C_1 \delta  + C_2  \frac{1}{N}.
\label{RegularityEtaEqn}
\end{eqnarray}
\end{lemma}

In particular, (\ref{RegularityEtaEqn}) implies that the regularity condition of Theorem 4.20 in \cite{Kurtz1975} (equivalently  condition B of Theorem 8.6 of Chapter 3 of \cite{EthierAndKurtz}) is satisfied. (See also Remark 8.7 B of Chapter of \cite{EthierAndKurtz} regarding replacing $\sup_N$ with $\lim_N$.) 

\subsection{Compact containment of the fluctuations process $\eta^N$}

The main result of this section is Lemma \ref{CompactcontainmentEtaProcessLemma} below.
\begin{lemma} \label{CompactcontainmentEtaProcessLemma}
If $J_{2} = 3 \ceil*{\frac{D}{2}} + 6$, then there is a constant $C<\infty$ such that
\begin{eqnarray}
\sup_{N \in \mathbb{N}}  \mathbb{E} \sup_{t \in [0,T]} \norm{\eta^N_t}^{2}_{-J_{2}}  < C.
\label{Eq:CompactcontainmentEtaBound}
\end{eqnarray}
In particular, the process $\{\eta_{\cdot}^{N}\}_{N\in\mathbb{N}}$ satisfies the compact containment condition in $W^{-J,2}(\Theta)$ with $J\geq J_{2}+1=3 \ceil*{\frac{D}{2}} + 7$.
\end{lemma}
\begin{proof}
The proof of this statement follows by the representation (\ref{DecompositionFinalEta}) together with the a-priori bounds of Lemma \ref{UniformEtaBoundLemma} and \ref{L:CompactContainmenrM_process}.

Let $\{ f_a \}_{a=1}^{\infty}$ be a complete orthonormal basis for $W_0^{J_{2},2}$  with $J_{2} = 3  \ceil*{\frac{D}{2}} + 6$. Equation  (\ref{DecompositionFinalEta}) with $f=f_{a}$ gives
\begin{eqnarray}
\mathbb{E}\sum_{a\geq 1}\sup_{t\in[0,T]}\la f_a, \eta_t^N \ra^2 &\leq& 2 \mathbb{E}\sum_{a\geq 1}\sup_{t\in[0,T]}\la f_a,  \Xi_t^N \ra^2 + 2 \mathbb{E}\sum_{a\geq 1}\sup_{t\in[0,T]}\la f_a, \sqrt{N} ( \tilde \mu_t^N - \bar \mu_t ) \ra^2.\nonumber
\end{eqnarray}

Following the arguments in equations (\ref{XiRegularity11})-(\ref{XiRegularity22}) with $r=0$  and using Lemma \ref{L:CompactContainmenrM_process} gives
\begin{align}
\mathbb{E} \sum_{a\geq 1}\sup_{t\in[0,T]} \bigg{[} \big{(} \la f_{a}, \Xi_t^N \ra -  \la f_{a}, \Xi_0^N \ra  \big{)}^2 \bigg{]} &\leq C  \bigg{[} \int_0^T \bigg{(} \sum_{a\geq 1}\norm{f_{a}}_L^2  \mathbb{E} \bigg{[} \norm{\Xi_s^N}_{-J_{1}}^2 \bigg{]} + \sum_{a\geq 1}\norm{f_{a}}_{J_1+1}^2  \mathbb{E} \bigg{[} \norm{ \Xi^N_s }_{ - J_1 }^2 \bigg{]} \bigg{)}  ds \notag \\
&\quad+   \sum_{a\geq 1}\norm{f_{a}}_L^2    +  N^{-1}  \sum_{a\geq 1} \norm{f_{a}}_L^2 \bigg{]}\nonumber\\
&\leq C  \bigg{[} \sup_{s\in[0,T]}\mathbb{E} \bigg{[} \norm{\Xi_s^N}_{-J_{1}}^2 \bigg{]} \sum_{a\geq 1}\norm{f_{a}}_L^2   +  \sup_{s\in[0,T]}\mathbb{E} \bigg{[} \norm{ \Xi^N_s }_{ - J_1 }^2 \bigg{]}\sum_{a\geq 1}\norm{f_{a}}_{J_1+1}^2  \notag \\
&\quad+   \sum_{a\geq 1}\norm{f_{a}}_L^2    +  N^{-1}  \sum_{a\geq 1} \norm{f_{a}}_L^2 \bigg{]}\nonumber
\end{align}

Similarly, using now (\ref{iidCLTbound}), we have
\begin{align}
 \mathbb{E}\sum_{a\geq 1}\sup_{t\in[0,T]}\la f_a, \sqrt{N} ( \tilde \mu_t^N - \bar \mu_t ) \ra^2&\leq C \sum_{a\geq 1}\norm{f_a}_{L}^{2}\nonumber
\end{align}

Putting the last displays together we obtain
\begin{align}
\mathbb{E}\sum_{a\geq 1}\sup_{t\in[0,T]}\la f_a, \eta_t^N \ra^2 &\leq  C \bigg{[} \sup_{s\in[0,T]}\mathbb{E} \bigg{[} \norm{\Xi_s^N}_{-J_{1}}^2 \bigg{]} \sum_{a\geq 1}\norm{f_{a}}_L^2   +  \sup_{s\in[0,T]}\mathbb{E} \bigg{[} \norm{ \Xi^N_s }_{ - J_1 }^2 \bigg{]}\sum_{a\geq 1}\norm{f_{a}}_{J_1+1}^2  \notag \\
&\quad+   \sum_{a\geq 1}\norm{f_{a}}_L^2    +  N^{-1}  \sum_{a\geq 1} \norm{f_{a}}_L^2 + \mathbb{E} \bigg{[} \norm{\Xi_0^N}_{-J_{2}}^2 \bigg{]} \bigg{]}.\nonumber
\end{align}

By Lemma \ref{UniformEtaBoundLemma} we have that $\sup_{N\in\mathbb{N}}\sup_{t\in[0,T]}\mathbb{E} \norm{\eta^N_{t}}^{2}_{-J_{1}}\leq C$. Since, $J_{2}>J_{1}+1+\frac{D}{2}>L+\frac{D}{2}$, we also obtain (by Sobolev embedding as before) that $\sum_{a\geq 1}\norm{f_{a}}_{J_1+1}^2<\infty$ and $\sum_{a\geq 1}\norm{f_{a}}_{L}^2<\infty$. In addition, since $J_{2}>J_{1}$ we have that $\norm{\cdot}_{-J_{2}}\leq C \norm{\cdot}_{-J_{1}}$ which then, due to (\ref{XiFinalBound}), leads to $\sup_{N\in\mathbb{N}}\mathbb{E} \bigg{[} \norm{\Xi_0^N}_{-J_{2}}^2 \bigg{]}<\infty$. Hence, we indeed have that
\begin{align}
\sup_{N \in \mathbb{N}}\mathbb{E}\sum_{a\geq 1}\sup_{t\in[0,T]} \la f_{a}, \eta_t^N \ra^{2}\leq C.\nonumber
\end{align}

Hence, by Parseval's identity we obtain
\begin{align}
\sup_{N \in \mathbb{N}}  \mathbb{E} \sup_{t \in [0,T]} \norm{\eta^N_t}_{-J_{2}}^2&=\sup_{N \in \mathbb{N}}\mathbb{E}\sup_{t\in[0,T]}\sum_{a\geq 1} \la f_{a}, \eta_t^N \ra^{2}\leq C.\nonumber
\end{align}

Now, due to the bound in the last display, we obtain that for every $\epsilon>0$, there is a constant $C_{\epsilon}$ such that
\begin{align}
\sup_{N \in \mathbb{N}}  \mathbb{P}\left\{ \sup_{t \in [0,T]} \norm{\eta^N_t}_{-J_{2}}^2>C_{\epsilon}\right\}&\leq \epsilon,\nonumber
\end{align}
and, due to the fact that the set $\left\{\phi\in W^{-(J_{2}+1),2}: \norm{\phi}_{-J_{2}}\leq C_{\epsilon}\right\}$ is a compact subset of $W^{-(J_{2}+1),2}$, we obtain the validity of the compact containment condition for $\{\eta_{\cdot}^{N}\}_{N\in\mathbb{N}}$ in $W^{-J,2}$ with $J\geq J_{2}+1$, as desired.
\end{proof}

\subsection{Relative Compactness of $\eta^N$ and $\sqrt{N} M^N$}

\begin{lemma} \label{RelativeCompactnessLemma}
Let $T>0$ and $J \geq 3  \ceil*{\frac{D}{2}} + 7$. Then, the sequences  $\{\mu^N_{t},t\in[0,T]\}_{N\in\mathbb{N}}$, $\{\eta^N_{t},t\in[0,T]\}_{N\in\mathbb{N}}$ and $\{\sqrt{N} M^N_{t},t\in[0,T]\}_{N\in\mathbb{N}}$ are relatively compact in $D_{\mathcal{M}(\mathbb{R}^{1 + d})}[0,T]$, $ D_{W^{-J,2} }([0,T])$ and $ D_{W^{-J,2} }([0,T])$ respectively.
\end{lemma}
\begin{proof}
Relative compactness of $\mu^N$ was proven in \cite{NeuralNetworkLLN}. Lemmas \ref{L:regularityEta}, \ref{CompactcontainmentEtaProcessLemma} for $\eta^{N}$ and Lemmas \ref{L:CompactContainmenrM_process}, \ref{L:RegularityMprocess} for $\sqrt{N} M^N$ combined with Theorem 8.6 of Chapter 3 of \cite{EthierAndKurtz} (and using Remark 8.7 B of \cite{EthierAndKurtz}), equivalently Theorem 4.20 of \cite{Kurtz1975}, prove the result.
\end{proof}


\section{Continuity properties and identification of the limiting equation} \label{Identification}
\begin{lemma}\label{L:ContinuityOfETA}
Let $J \geq 3  \ceil*{\frac{D}{2}} + 7$. Any limit point of $\{ \eta_t^N, t \in [0,T] \}_{N \in \mathbb{N}}$ is continuous, i.e., it takes values in $C_{W^{-J,2}}([0,T])$.
\end{lemma}
\begin{proof}
In order to prove that any limit point of $\{ \eta_t^N, t \in [0,T] \}_{N \in \mathbb{N}}$ takes values in $C_{W^{-J,2}}([0,T])$, it is sufficient to show that
\begin{eqnarray*}
\lim_{N \rightarrow \infty}  \mathbb{E}\left[   \sup_{ t \leq T}  \norm{ \eta_t^N - \eta_{t^{-}}^N }_{-J}^2  \right]= 0.
\end{eqnarray*}

We again use the decomposition (\ref{DeCompositionEtaRepeat}),
\begin{eqnarray}
   \sup_{ t \leq T}  \norm{ \eta_t^N - \eta_{t^{-}}^N }_{-J}^2 \leq    2 \sup_{ t \leq T}  \norm{ \Xi_t^N - \Xi_{t^{-}}^N }_{-J}^2  + 2 \sup_{ t \leq T}  \norm{ \sqrt{N} ( \tilde \mu_t^N - \bar \mu_t) - \sqrt{N} ( \tilde \mu_{t^{-}}^N   - \bar \mu_{t^{-}} ) }_{-J}^2.\notag
\end{eqnarray}
Since both $\tilde \mu_t^N$ and $\bar \mu_t$ are continuous, $ \norm{ (\tilde \mu_t^N - \bar \mu_t) - ( \tilde \mu_{t^{-}}^N   - \bar \mu_{t^{-}} ) }_{-J} = 0$.

Next, let $\{ f_a \}_{a=1}^{\infty}$ be a complete orthonormal basis for $W_0^{J,2}$. As it follows by (\ref{Eq:XiRepresentation}) the discontinuities of $\la f_{a}, \Xi_t^N \ra$ are those of $\sqrt{N} \la f_{a}, M_t^N \ra$ and $R_t^{1,N} + R_t^{2,N}$. Hence,  we shall have,
\begin{eqnarray}
\la f_{a}, \Xi_t^N \ra - \la f_{a},  \Xi_{t^{-}}^N \ra = \sqrt{N} \la f_{a}, M_t^N \ra - \sqrt{N} \la f_{a}, M_{t^{-}}^N \ra     +  R_{t}^{N} - R_{t^{-}}^N,\notag
\end{eqnarray}
where $R_t^N = R_t^{1,N} + R_t^{2,N}$.

Note that $\la f_{a}, M_t^N \ra $ is a pure jump process where the size of the $k$-th jump is bounded by
\begin{eqnarray}
&\phantom{.}& \bigg{|} \frac{1}{N} \alpha \big{(} y_k -  \la c \sigma(w \cdot x_k),  \nu^N_k \ra \big{)} \la \sigma(w \cdot x_k) \partial_c f_{a}, \nu_k^N \ra   - D^{1,N}_k  \bigg{|}  \notag \\
&+& \bigg{|}  \frac{1}{N} \alpha \big{(} y_k -  \la c \sigma(w \cdot x_k),  \nu^N_k \ra \big{)} \la c  \sigma'(w \cdot x_k) x \cdot \nabla_w f_{a}, \nu_k^N \ra  - D^{2,N}_k \bigg{|}.
\end{eqnarray}

Therefore, for $0 \leq t \leq T$,
\begin{align}
&\bigg{(} \sqrt{N}\la f_{a}, M_t^N \ra - \sqrt{N}\la f_{a}, M_{t^{-}}^N \ra  \bigg{)}^2 \leq\nonumber\\
&\qquad\leq 2  N \sup_{0 \leq k  \leq \floor*{N t} -1}  \bigg{(}  \frac{1}{N} \alpha \big{(} y_k -  \la c \sigma(w \cdot x_k),  \nu^N_k \ra \big{)} \la \sigma(w \cdot x_k) \partial_c f_{a}, \nu_k^N \ra   - D^{1,N}_k \bigg{)}^2 \notag \\
&\qquad+ 2 N\sup_{ 0 \leq k  \leq \floor*{N t} -1}  \bigg{(}  \frac{1}{N} \alpha \big{(} y_k -  \la c \sigma(w \cdot x_k),  \nu^N_k \ra \big{)} \la c  \sigma'(w \cdot x_k) x \cdot \nabla_w f_{a}, \nu_k^N \ra  - D^{2,N}_k \bigg{)}^2. \notag
\end{align}

Due to the uniform bound (\ref{UniformBoundfromLLNpaper}), the bound (\ref{SupToJnorm}), and $\pi(dx,dy)$ having compact support,
\begin{align}
\big{|} \la f_{a}, \sqrt{N} M_t^N \ra - \la f_{a}, \sqrt{N} M_{t^{-}}^N \ra  \big{|}^2  &\leq  \frac{C}{N} \bigg{(} \sum_{ | \alpha| =1} \sup_{(c,w) \in K} | D^{\alpha} f_{a}(c,w) |  \bigg{)}^2 \leq \frac{C}{N} \norm{f_{a}}_L^2.\notag
\end{align}

Similarly,
\begin{eqnarray}
\bigg{(} R_{t}^{N} - R_{t^{-}}^N \bigg{)}^2 \leq \frac{C}{N} \norm{f_{a}}_L^2.\notag
\end{eqnarray}

Therefore, for $0 \leq t \leq T$,
\begin{eqnarray}
\la f_{a}, \Xi_t^N - \Xi_{t^{-}}^N \ra^2 \leq \frac{C}{N} \norm{f_{a}}_L^2.\notag
\end{eqnarray}

Since  $J-L >D/2$, the embedding $ W_0^{J,2}(\Theta) \hookrightarrow W_0^L(\Theta)$ is of Hilbert-Schmidt type (Theorem 6.53 of \cite{Adams}) and we have the bound $\sum_{a} \norm{ f_a }_L^2 < \infty$. Hence, we obtain
\begin{eqnarray}
\mathbb{E}\left[\sup_{ t \leq T}  \norm{ \Xi_t^N - \Xi_{t^{-}}^N }_{-J}^2\right]  \leq \frac{C}{N}.\notag
\end{eqnarray}

Consequently, $\lim_{N \rightarrow \infty}   \mathbb{E}\left[ \sup_{ t \leq T}  \norm{ \eta_t^N - \eta_{t^{-}}^N }_{-J}^2\right] = 0$, concluding the proof of the lemma.
\end{proof}

\begin{lemma} \label{VarianceOfBarM}
Let $J_{1} = 2  \ceil*{\frac{D}{2}} + 4$ and for $(x,y)\in\mathcal{X}\times\mathcal{Y}$, $\mu\in\mathcal{M}(\mathbb{R}^{1+d})$ and $h\in\mathcal{C}^{1}_{0}(\mathbb{R}^{1+d})$ define the operator
\[
\mathcal{R}_{x,y,\mu}[h]=(y - \la c \sigma(w \cdot x), \mu \ra )  \la \nabla(c  \sigma(w \cdot x) ) \cdot  \nabla h, \mu \ra.
\]

Then, for every $f \in W_0^{J_{1},2}(\Theta)$, $\sqrt{N} \la f, M_t^N \ra \in D_{\mathbb{R}}([0,T])$ converges in distribution to a distribution valued mean-zero Gaussian martingale $\bar M_t$ with variance
\begin{eqnarray*}
\textrm{Var} \bigg{[} \la f, \bar M_t \ra \bigg{]}
 &=&\alpha^2 \int_0^t \bigg{[}  \int_{\mathcal{X}\times\mathcal{Y}} \bigg{(} \mathcal{R}_{x,y,\bar\mu_{s}}[f] -\int_{\mathcal{X}\times\mathcal{Y}}  \mathcal{R}_{x,y,\bar\mu_{s}}[f]\pi(dx,dy)\bigg{)}^2  \pi(dx,dy)\bigg{]} ds \notag
\end{eqnarray*}

More generally, for every $f,g \in W_0^{J_{1},2}(\Theta)$, $(\sqrt{N} \la f, M_t^N \ra, \sqrt{N} \la g, M_t^N \ra ) \in D_{\mathbb{R}^2}([0,T])$ converges to a distribution valued mean-zero Gaussian martingale with covariance function
\begin{eqnarray}
\textrm{Cov} \bigg{[} \la f, \bar M_{t} \ra, \la g, \bar M_{t} \ra \bigg{]} &=& \alpha^2 \int_0^{t} \bigg{[}  \int_{\mathcal{X}\times\mathcal{Y}} \bigg{(} \mathcal{R}_{x,y,\bar\mu_{s}}[f] -\int_{\mathcal{X}\times\mathcal{Y}}  \mathcal{R}_{x,y,\bar\mu_{s}}[f]\pi(dx,dy)\bigg{)}\times\nonumber\\
& &\qquad\qquad\times
\bigg{(} \mathcal{R}_{x,y,\bar\mu_{s}}[g] -\int_{\mathcal{X}\times\mathcal{Y}}  \mathcal{R}_{x,y,\bar\mu_{s}}[g]\pi(dx,dy)\bigg{)}
  \pi(dx,dy)\bigg{]} ds.\notag
\end{eqnarray}

\end{lemma}
\begin{proof}
Recall that
\begin{eqnarray}
\sqrt{N} \la f, M_t^N \ra &=& N^{1/2} \sum_{k=0}^{ \floor*{N t} -1} \bigg{(}   \frac{\alpha}{N} \big{(} y_k -  \la c \sigma(w \cdot x_k),  \nu^N_k \ra \big{)} \la \sigma(w \cdot x_k) \partial_c f, \nu_k^N \ra   - D^{1,N}_k \bigg{)} \notag \\
& &+ N^{1/2} \sum_{k=0}^{ \floor*{N t}-1 } \bigg{(}  \frac{\alpha}{N} \big{(} y_k -  \la c \sigma(w \cdot x_k),  \nu^N_k \ra \big{)} \la c  \sigma'(w \cdot x_k) x \cdot \nabla_w f, \nu_k^N \ra  - D^{2,N}_k \bigg{)} \notag \\
&=& \sum_{k=0}^{ \floor*{N t}-1 }  X^N_k,\nonumber
\end{eqnarray}
where we can write
\begin{eqnarray}
X^N_k & \vcentcolon = & 
\frac{\alpha}{\sqrt{N}}\left[\big{(} y_k -  \la c \sigma(w \cdot x_k),  \nu^N_k \ra \big{)} \la \nabla(c  \sigma(w \cdot x_k)) \cdot \nabla f, \nu_k^N \ra\right.\nonumber\\
& &-\left.\left(\int_{\mathcal{X}\times\mathcal{Y}}  (y - \la c \sigma(w \cdot x), \nu^N_k \ra )  \la \nabla (c  \sigma(w \cdot x) ) \cdot  \nabla f, \nu_k^N \ra   \pi(dx,dy)\right)\right].\notag
\end{eqnarray}

Due to the compact support of $\pi(dx,dy)$ and the uniform bound $|c^i| + \norm{w^i} < C_{o}$, $|X^N_k | \leq C N^{-1/2}$.

$\sqrt{N} \la f, M_t^N \ra$ is a pure jump process and its quadratic variation is
\begin{eqnarray}
\left[\sqrt{N} \la f, M_{\cdot}^N \ra  \right]_{t} &=& \sum_{k=0}^{ \floor*{N t}-1 }  ( X^N_k )^2 \notag\\
  &=&  \sum_{k=0}^{ \floor*{N t}-1 } \mathbb{E} \big{[}  ( X^N_k )^2  \big{|} \mathcal{F}_k^{N} \big{]} +   \sum_{k=0}^{ \floor*{N t}-1 } \bigg{(} ( X_k^N )^2 -  \mathbb{E} \big{[}  ( X^N_k )^2  \big{|} \mathcal{F}_k^{N} \big{]} \bigg{)}.
\label{QuadraticVariationMpreLimit}
\end{eqnarray}

The first term on the right hand side of (\ref{QuadraticVariationMpreLimit}) becomes:
\begin{eqnarray}
&\phantom{.}& \sum_{k=0}^{ \floor*{N t}-1 }  \mathbb{E} \bigg{[} ( X^N_k )^2  \big{|} \mathcal{F}_{k}^{N} \bigg{]}  = \frac{\alpha^2}{N} \sum_{k=0}^{ \floor*{N t}-1 }  \bigg{[}  \int_{\mathcal{X}\times\mathcal{Y}} \bigg{(} (y - \la c \sigma(w \cdot x), \nu^N_k \ra )  \la \nabla (c  \sigma(w \cdot x) ) \cdot  \nabla f, \nu_k^N \ra \bigg{)}^2  \pi(dx,dy)          \notag \\
 & &-   \left(\int_{\mathcal{X}\times\mathcal{Y}}  (y - \la c \sigma(w \cdot x), \nu^N_k \ra )  \la \nabla (c  \sigma(w \cdot x) ) \cdot  \nabla f, \nu_k^N \ra   \pi(dx,dy)\right)^{2}\bigg{]}\nonumber\\
&=&  \alpha^2 \int_0^t  \int_{\mathcal{X}\times\mathcal{Y}} \bigg{(} (y - \la c \sigma(w \cdot x), \mu^N_s \ra )  \la \nabla(c  \sigma(w \cdot x) ) \cdot  \nabla f, \mu_s^N \ra \bigg{)}^2  \pi(dx,dy)   ds \nonumber\\
   & &\quad-\alpha^2 \int_0^t \left( \int_{\mathcal{X}\times\mathcal{Y}}  (y - \la c \sigma(w \cdot x), \mu^N_s \ra )  \la \nabla(c  \sigma(w \cdot x) ) \cdot  \nabla f, \mu_s^N \ra   \pi(dx,dy) \right)^{2}  ds+ \mathcal{O}(N^{-1})\nonumber\\
&=& \alpha^2 \int_0^t  \int_{\mathcal{X}\times\mathcal{Y}} \mathcal{R}^{2}_{x,y,\mu^{N}_{s}}[f]  \pi(dx,dy)   ds -
\alpha^2 \int_0^t \left( \int_{\mathcal{X}\times\mathcal{Y}}  \mathcal{R}_{x,y,\mu^{N}_{s}}[f]   \pi(dx,dy) \right)^{2}  ds+ \mathcal{O}(N^{-1})  \nonumber\\
&=& \alpha^2 \int_0^t  \int_{\mathcal{X}\times\mathcal{Y}}\left( \mathcal{R}_{x,y,\mu^{N}_{s}}[f]     -
  \int_{\mathcal{X}\times\mathcal{Y}}  \mathcal{R}_{x,y,\mu^{N}_{s}}[f]   \pi(dx,dy) \right)^{2}\pi(dx,dy)  ds+ \mathcal{O}(N^{-1})
\label{Xmda2}
 \end{eqnarray}

The second term on the right hand side of (\ref{QuadraticVariationMpreLimit}) can be bounded as follows:
\begin{eqnarray}
&\phantom{.}& \mathbb{E} \bigg{[} \bigg{(}  \sum_{k=0}^{ \floor*{N t}-1 } \bigg{[} ( X_k^N )^2 -  \mathbb{E} \big{[}  ( X^N_k )^2  \big{|} \mathcal{F}_k^{N} \big{]} \bigg{]} \bigg{)}^2 \bigg{]} \notag \\
&=&  \sum_{j=0}^{   \floor*{N t}-1 }   \sum_{k=0}^{ \floor*{N t}-1 }  \mathbb{E} \bigg{[} \bigg{(} ( X_k^N )^2 -  \mathbb{E} \big{[}  ( X^N_k )^2  \big{|} \mathcal{F}_k^{N} \big{]} \bigg{)}  \bigg{(} ( X_j^N )^2 -  \mathbb{E} \big{[}  ( X^N_j )^2  \big{|} \mathcal{F}_j^{N} \big{]} \bigg{)}  \bigg{]}  \notag \\
&=&  \sum_{k=0}^{ \floor*{N t}-1 }  \mathbb{E} \bigg{[}  \bigg{(} ( X_k^N )^2 -  \mathbb{E} \big{[}  ( X^N_k )^2  \big{|} \mathcal{F}_k^{N} \big{]} \bigg{)}^2 \bigg{]} \notag \\
& &+  2 \sum_{j=0}^{   \floor*{N t}-2 }   \sum_{k=j+1}^{ \floor*{N t}-1 }  \mathbb{E} \bigg{[} \mathbb{E} \bigg{[} ( X_k^N )^2 -  \mathbb{E} \big{[}  ( X^N_k )^2  \big{|} \mathcal{F}_k^{N} \big{]} \bigg{|} \mathcal{F}_{k}^{N} \bigg{]}  \bigg{(} ( X_j^N )^2 -  \mathbb{E} \big{[}  ( X^N_j )^2  \big{|} \mathcal{F}_j^{N} \big{]} \bigg{)}  \bigg{]} \notag \\
&=&  \sum_{k=0}^{ \floor*{N t}-1 }  \mathbb{E} \bigg{[}  \bigg{(} ( X_k^N )^2 -  \mathbb{E} \big{[}  ( X^N_k )^2  \big{|} \mathcal{F}_k^{N} \big{]} \bigg{)}^2 \bigg{]}  \notag \\
&\leq&  C \sum_{k=0}^{ \floor*{N t}-1 }  N^{-2} \leq \frac{C}{N},\nonumber
\end{eqnarray}
where the last inequality uses the bound $|X^N_k | \leq C N^{-1/2}$.

Therefore, since $\mu^N \overset{p} \rightarrow \bar \mu$ in $D_E([0,T])$ and by applying the continuous mapping theorem to (\ref{Xmda2}), we have that for each $t \in [0,T]$,
\begin{eqnarray}
&\phantom{.}& \left[ \sqrt{N} \la f, M_{\cdot}^N \ra  \right]_{t}  \overset{p} \rightarrow
\alpha^2 \int_0^t  \int_{\mathcal{X}\times\mathcal{Y}}\left( \mathcal{R}_{x,y,\bar{\mu}_{s}}[f]     -
  \int_{\mathcal{X}\times\mathcal{Y}}  \mathcal{R}_{x,y,\bar{\mu}_{s}}[f]   \pi(dx,dy) \right)^{2}\pi(dx,dy)  ds
 \label{QuadraticVariationMFinalLimit}
\end{eqnarray}
as $N \rightarrow \infty$.

Using the same approach as in Lemma \ref{L:ContinuityOfETA}, we also have that
\begin{eqnarray}
\lim_{N \rightarrow \infty}  \mathbb{E} \left[   \sup_{ t \leq T} \bigg{|}  \sqrt{N} \la f, M_t^N \ra  -  \sqrt{N} \la f, M_{t^{-}}^N \ra \bigg{|}   \right]= 0.
\label{Mcontinuity2222}
\end{eqnarray}

The first statement of this lemma follows from (\ref{QuadraticVariationMFinalLimit}), (\ref{Mcontinuity2222}), and Theorem 7.1.4 of \cite{EthierAndKurtz}. The convergence of $(\sqrt{N} \la f, M_t^N \ra, \sqrt{N} \la g, M_t^N \ra )$ follows by a similar procedure and the Cramer-Wold theorem.

\end{proof}

\begin{lemma} \label{LemmaIdentifyLimit}
Let $J \geq 3 \ceil*{\frac{D}{2}} + 7$. Any limit point $\bar \eta$ must satisfy the stochastic evolution equation
\begin{eqnarray}
\la f, \bar  \eta_t \ra &=& \la f, \bar \eta_0 \ra + \int_0^t  \bigg{(} \int_{\mathcal{X}\times\mathcal{Y}}  \alpha \big{(} y -  \la c \sigma(w \cdot x), \bar \mu_{s} \ra \big{)} \la \sigma(w \cdot x) \partial_c f, \bar \eta_{s} \ra \pi(dx,dy)  \bigg{)} ds\notag \\
& &- \alpha \int_0^t  \bigg{(} \int_{\mathcal{X}\times\mathcal{Y}}   \la c \sigma(w \cdot x), \bar \eta_s \ra \la \sigma(w \cdot x) \partial_c f, \bar \mu_{s} \ra \pi(dx,dy)  \bigg{)} ds\notag \\
& &+ \int_0^t \bigg{(} \int_{\mathcal{X}\times\mathcal{Y}}   \alpha \big{(} y -  \la c \sigma(w \cdot x),  \bar \mu_{s} \ra \big{)} \la c  \sigma'(w \cdot x) x \cdot \nabla_w f, \bar \eta_{s} \ra \pi(dx, dy) \bigg{)}ds \notag \\
& &- \alpha \int_0^t \bigg{(} \int_{\mathcal{X}\times\mathcal{Y}}  \la c \sigma(w \cdot x),  \bar \eta_{s} \ra \big{)} \la c  \sigma'(w \cdot x) x \cdot \nabla_w f, \bar \mu_{s} \ra \pi(dx, dy) \bigg{)}ds \notag \\
& &+ \la f, \bar M_t \ra,
\label{SPDEeta}
\end{eqnarray}
for every $f \in W_0^{J,2}(\Theta)$.
\end{lemma}
\begin{proof}
The result can be proven by considering the pre-limit evolution equation (\ref{EtaEqn1}). For each $f \in W_0^{J,2}(\Theta)$, $\displaystyle \sup_{t \in [0,T]} R_t^N \overset{p} \rightarrow 0$. Due to the uniform bound $\displaystyle \sup_{N \in \mathbb{N}} \sup_{t \in [0,T]} \mathbb{E}[  \norm{ \eta_t }_{-J_{1}}^2 ] < C$, it can be shown that $\Gamma^{1,N}_t\overset{p} \rightarrow 0$ and $\Gamma^{2,N}_t \overset{p} \rightarrow 0$ uniformly in $t \in [0,T]$. Indeed, recall that
\begin{eqnarray}
 \Gamma^{1,N}_t &=& \frac{1}{\sqrt{N}} \int_0^t  \int_{\mathcal{X}\times\mathcal{Y}} - \alpha   \la c \sigma(w \cdot x), \eta_s^N \ra \la \sigma(w \cdot x) \partial_c f, \eta^N_{s} \ra \pi(dx,dy)   ds \notag \\
\Gamma^{2,N}_t &=& \frac{1}{\sqrt{N}} \int_0^t  \int_{\mathcal{X}\times\mathcal{Y}}  -\alpha   \la c \sigma(w \cdot x), \eta_s^N \ra  \la c \sigma'(w \cdot x) x \cdot \nabla_w f, \eta^N_{s} \ra \pi(dx,dy) ds.\notag
\end{eqnarray}

Recall that $J_1 = 2  \ceil*{\frac{D}{2}} + 4$. Then, using the compactness of $\mathcal{X} \times \mathcal{Y}$, the bound (\ref{UnfiormEtaBound}), and Young's inequality
\begin{eqnarray}
\mathbb{E}  \bigg{[} \sup_{t\in[0,T]}\big{|}  \Gamma^{1,N}_t  \big{|}  \bigg{]} &\leq& \mathbb{E} \bigg{[} \frac{1}{\sqrt{N}} \int_0^T  \int_{\mathcal{X}\times\mathcal{Y}} \norm{ c \sigma (w \cdot x) }_{J_{1}} \norm{ \eta_s^N}_{-J_{1}} \norm{  \sigma(w \cdot x) \partial_c f}_{J_1} \norm{\eta^N_{s}}_{-J_1} \pi(dx,dy) ds \bigg{]} \notag \\
&\leq&  \frac{C}{\sqrt{N}}   \norm{  f}^{2}_{J_1+1}  \leq \frac{C}{\sqrt{N}}.\notag
\end{eqnarray}

Similarly, $ \mathbb{E} \bigg{[} \sup_{t\in[0,T]}\big{|} \Gamma^{2,N}_t  \big{|} \bigg{]} \leq \frac{C}{\sqrt{N}}$.  Therefore, $\Gamma^{1,N}_t\overset{p} \rightarrow 0$ and $\Gamma^{2,N}_t \overset{p} \rightarrow 0$ uniformly in $t \in [0,T]$.

By Lemma \ref{RelativeCompactnessLemma} we have that the sequence $(\mu^{N}_t, \eta^{N}_{t},\sqrt{N}M^{N}_{t})$ is relatively compact in $D_{\mathcal{M}(\mathbb{R}^{1 + d})\times W^{-J,2}\times W^{-J,2}}[0,T]$. Denoting by $(\bar{\mu}_{t},\bar{\eta}_{t},\bar{M}_{t})$ a limiting point of an appropriate subsequence and due to the linearity of the involved operators in (\ref{EtaEqn1}) we obtain by Theorem 5.5 in \cite{KurtzAndProtter} and Lemma \ref{VarianceOfBarM}
that  $\bar \eta$ satisfies (\ref{SPDEeta}).
\end{proof}

\section{Uniqueness of the stochastic evolution equation} \label{Uniqueness}

The limiting distribution $\bar \eta_t$ satisfies the stochastic evolution equation (\ref{SPDEeta}). Suppose (\ref{SPDEeta}) does not have a unique solution. Then, there are at least two solutions $\bar \eta^1$ and $\bar \eta^2$ which satisfy (\ref{SPDEeta}). Define $\Phi_t = \bar \eta^1_t - \bar \eta^2_t$. Our goal is to show that $\|\Phi_{t}\|_{-J}=0$ for all $t\leq T$. $\Phi_t$ satisfies the deterministic equation

\begin{eqnarray}
\la f, \Phi_t \ra &=& \alpha \int_0^t \int_{\mathcal{X} \times \mathcal{Y}} \pi(dx,dy) \bigg{[} \big{(} y - \la c \sigma(w x), \bar \mu_s \ra  \big{)} \la c \sigma'(w x) x \cdot \nabla_{w}f, \Phi_s \ra + \big{(} y - \la c \sigma(w x), \bar \mu_s \ra  \big{)} \la \sigma(w x) \partial_{c}f, \Phi_s \ra  \bigg{]} ds \notag \\
&-& \alpha \int_0^t \int_{\mathcal{X}\times \mathcal{Y}} \pi(dx,dy) \bigg{[} \la c \sigma(w x), \Phi_s \ra  \la c \sigma'(w x) x \cdot \nabla_{w}f, \bar \mu_s \ra + \la c \sigma(w x),  \Phi_s \ra \la \sigma(w x) \partial_{c}f, \bar \mu_s \ra  \bigg{]} ds, \notag \\
\la f, \Phi_0 \ra &=& 0.\notag
\end{eqnarray}

Therefore,
\begin{eqnarray}
 \la f, \Phi_t \ra^2 &=& 2\alpha \int_0^t \int_{\mathcal{X} \times \mathcal{Y}} \pi(dx,dy)  \bigg{[} \big{(} y - \la c \sigma(w x), \bar \mu_s \ra  \big{)} \la c \sigma'(w x) x \cdot \nabla_{w}f, \Phi_s \ra \nonumber\\
 & &\hspace{5cm}+ \big{(} y - \la c \sigma(w x), \bar \mu_s \ra  \big{)} \la \sigma(w x)  \partial_{c}f, \Phi_s \ra  \bigg{]}  \la f, \Phi_s \ra ds \notag \\
&-&2\alpha \int_0^t \int_{\mathcal{X} \times \mathcal{Y}} \pi(dx,dy) \bigg{[}  \la c \sigma(w x), \Phi_s \ra   \la c \sigma'(w x) x \cdot \nabla_{w}f, \bar \mu_s \ra + \la c \sigma(w x),  \Phi_s \ra   \la \sigma(w x)  \partial_{c}f, \bar \mu_s \ra  \bigg{]}  \la f, \Phi_s \ra ds. \notag
\label{DiffPhi}
\end{eqnarray}

Using Young's inequality, the fact that $\bar \mu$ takes values in a compact set, $\pi(dx,dy)$ has compact support, and the bound (\ref{SupToJnorm}),
\begin{eqnarray}
 \la f, \Phi_t \ra^2 &\leq& \alpha \int_0^t \int_{\mathcal{X} \times \mathcal{Y}} \pi(dx,dy)  \bigg{[} \big{(} y - \la c \sigma(w x), \bar \mu_s \ra  \big{)} \la c \sigma'(w x) x \nabla_{w}f, \Phi_s \ra \nonumber\\
 & &\hspace{5cm}+ \big{(} y - \la c \sigma(w x), \bar \mu_s \ra  \big{)} \la \sigma(w x)  \partial_{c}f, \Phi_s \ra  \bigg{]}  \la f, \Phi_s \ra ds \notag \\
& &+ C \int_0^t \big{(}  \la f, \Phi_s \ra^2  + \norm{f}_{L}^2 \norm{ \Phi_s}_{-J}^2 \big{)} ds,
\label{DiffPhi2}
\end{eqnarray}
where $L = \ceil*{\frac{D}{2}} + 3$ and $J \geq 3  \ceil*{\frac{D}{2}} + 7$.

\begin{lemma} \label{CompactLemmaEta}
For any $f \in W_0^{J,2}(\Theta)$ and every $t \in [0,T]$,
\begin{eqnarray}
\la f, \bar{\eta}_t \ra = \la  b f, \bar{\eta}_t \ra,\notag
\end{eqnarray}
where $b$ is the bump function defined in equation (\ref{BumpFunctionDefinition}).
\end{lemma}
\begin{proof}
From Lemma \ref{CompactLemmaXi}, there exists a bump function $b(c,w)$ such that, for any $f \in W_0^{J,2}(\Theta)$ and every $t \in [0,T]$,
\begin{eqnarray}
\la f, \eta^N_t \ra = \la b f, \eta^N_t \ra.\notag
\end{eqnarray}
Furthermore, $b f \in C_c^{\infty}$.  Therefore, for all $N \in \mathbb{N}$,
\begin{eqnarray}
\sup_{t \in [0,T]} | \la f, \eta^N_t \ra - \la b f, \eta^N_t \ra | = 0.
\label{SupFBfiniteN}
\end{eqnarray}

Due to relative compactness, there is a sub-sequence
\begin{eqnarray}
\bigg{(} \la f, \eta^{N_k}_\cdot \ra, \la b f, \eta^{N_k}_\cdot \ra, \eta^{N_k}_{\cdot}, \sqrt{N} M^{N_k}_{\cdot} \bigg{)} \overset{d} \rightarrow \bigg{(} \la f, \bar \eta_{\cdot} \ra, \la b f, \bar \eta_{\cdot} \ra, \bar \eta_{\cdot}, \bar M_{\cdot} \bigg{)}.\notag
\end{eqnarray}
in $D_{\mathbb{R} \times \mathbb{R} \times W^{-J,2} \times W^{-J,2}}([0,T])$. Due to (\ref{SupFBfiniteN}), any limit point must satisfy $\la f, \bar\eta_t \ra = \la b f, \bar\eta_t \ra$ for each $t \in [0,T]$.

\end{proof}

Due to Lemma \ref{CompactLemmaEta}, we can re-write equation (\ref{DiffPhi2}) as
\begin{eqnarray}
 \la f, \Phi_t \ra^2 &\leq& \int_0^t \int_{\mathcal{X} \times \mathcal{Y}} \pi(dx,dy)  \bigg{[} \big{(} y - \la c \sigma(w x), \bar \mu_s \ra  \big{)} \la b(c,w)  c \sigma'(w x) x \cdot \nabla_{w}f, \Phi_s \ra \notag \\
 & &+ \big{(} y - \la c \sigma(w x), \bar \mu_s \ra  \big{)} \la b(c,w)\sigma(w x)  \partial_{c}f, \Phi_s \ra  \bigg{]}  \la f, \Phi_s \ra ds \notag \\
& &+ C \int_0^t \big{(}  \la f, \Phi_s \ra^2  + \norm{f}_{L}^2 \norm{ \Phi_s}_{-J}^2 \big{)} ds.
\label{DiffPhi3}
\end{eqnarray}

Let $\{ f_a \}_{a=1}^{\infty}$ be a complete orthonormal basis for $W_0^{J,2}$ where $J \geq 3 \ceil*{\frac{D}{2}} + 7$. Let $f = f_a$ in equation (\ref{DiffPhi3}) and then sum (\ref{DiffPhi2}) over all $a$. By Parseval's identity,

%

\begin{eqnarray}
\norm{ \Phi_t }_{-J}^2 &\leq& \int_0^t  \int \pi(dx,dy) \bigg{[} \big{(} y - \la c \sigma(w x), \bar \mu_s \ra  \big{)} \la \Phi_s, \mathcal{G}_1^{\ast} \Phi_s  \ra_{-J} + \big{(} y - \la c \sigma(w x), \bar \mu_s \ra  \big{)} \la \Phi_s, \mathcal{G}_2^{\ast} \Phi_s \ra_{-J}  \bigg{]}  ds \notag \\
& &+ C \int_0^t \norm{\Phi_s}_{-J}^2 ds.\notag
\end{eqnarray}

The operators $\mathcal{G}^1$ and $\mathcal{G}^2$ are defined in equation (\ref{operatorsG}). Since $\bar \mu_t$ takes values in a compact set and $\pi(dx,dy)$ has compact support,
\begin{eqnarray}
\norm{ \Phi_t }_{-J}^2 &\leq&  C_1 \int_0^t \int_{\mathcal{X} \times \mathcal{Y}} \pi(dx,dy) \bigg{(} \big{|} \la \Phi_s, \mathcal{G}_1^{\ast} \Phi_s  \ra_{-J} \big{|} + \big{|} \la \Phi_s, \mathcal{G}_2^{\ast} \Phi_s \ra_{-J} \big{|} \bigg{)}  ds \notag \\
& &+ C_2 \int_0^t \norm{\Phi_s}_{-J}^2 ds.\notag
\end{eqnarray}

Using Lemma \ref{LemmaBoundPsiG} and the fact that $\mathcal{X} \times \mathcal{Y}$ is a compact set,
\begin{eqnarray*}
\norm{ \Phi_t }_{-J}^2 &\leq&  C \int_0^t  \norm{\Phi_s}_{-J}^2 ds,
\end{eqnarray*}
which then by Gronwall's inequality gives $\norm{\Phi_t}_{-J}^2 = 0$ for $t \in [0,T]$. Thus, we have established the following result.

\begin{theorem} \label{TheoremUniquenessEtaBar}
Let $J\geq 3\ceil*{ \frac{D}{2} }+7$ with $D=d+1$. Then, the solution $\bar{\eta}$ to the stochastic evolution equation (\ref{SPDEeta}) is unique in $W^{-J,2}$.
\end{theorem}

\section{Proof of the Main Result} \label{ProofOfMainResult}
We now collect our results and prove Theorem \ref{MainTheoremCLT}. By Lemma \ref{RelativeCompactnessLemma} we have that the sequence $(\mu^{N}_t, \eta^{N}_{t},\sqrt{N}M^{N}_{t})$ is relatively compact in $D_{\mathcal{M}(\mathbb{R}^{1 + d}) \times W^{-J,2}\times W^{-J,2}}  ([0,T]) $.
 Lemma \ref{LemmaIdentifyLimit} establishes that the limit point satisfies the SPDE (\ref{SPDEmain}) and Theorem \ref{TheoremUniquenessEtaBar} proves that limit point is unique. Therefore, by Prokhorov's Theorem,  $\eta^N \overset{d} \rightarrow  \bar \eta$ in $D_{W^{-J,2}}([0,T])$ where $\bar \eta$ satisfies the stochastic evolution equation (\ref{SPDEmain}).

\section{Conclusion} \label{Conclusion}

Neural networks are nonlinear  models whose parameters are estimated from data using stochastic gradient descent.  They have achieved immense practical success over the past decade in a variety of applications in image, speech, and text recognition. However, there is limited mathematical understanding of their properties. This paper studies neural networks with a single hidden layer in the asymptotic regime of large network sizes and large numbers of stochastic gradient descent iterations. We rigorously prove a central limit theorem (CLT) for the empirical distribution of the neural network parameters. The limiting fluctuations process satisfies a stochastic partial differential equation and has Gaussian distribution.

\appendix
\section{Proof of Lemma \ref{L:RemainderTerms}}\label{A:Appendix0}
\begin{proof}[Proof of Lemma \ref{L:RemainderTerms}]
Let us recall that $\tilde R_t^{1,N}$ is the remainder term
\begin{eqnarray}
\tilde R_t^{1,N} = \sum_{k =0}^{\floor*{Nt}-1} \bigg{(} \big{(} \la f, \Xi_{\frac{k+1}{N}^{-}}^N \ra  + G_k^N  N^{-3/2} \big{)}^2  - \la f, \Xi^N_{\frac{k+1}{N}^{-}} \ra^2  \bigg{)},\notag
\end{eqnarray}
where $|G_k^N| < C \displaystyle \sum_{|\alpha| =2 } \sup_{c,w \in K} |D^{\alpha} f(c,w)  |$ due to the bound $| c_k^i| + \norm{w_k^i} < C_{o}$ and $\pi(dx,dy)$ having compact support. $K \subset \mathbb{R}^{1+d}$ is a compact set.

By the Sobolev embedding Theorem (Theorem 6.2 in \cite{Adams}), we have that
\begin{eqnarray}
 \sum_{|\alpha| \leq 2 } \sup_{c,w \in K} |D^{\alpha} f(c,w)  |  \leq   C\norm{ f}_{L}\nonumber
\end{eqnarray}
where $L = \ceil*{\frac{D}{2}} + 3$.

Therefore,
\begin{eqnarray}
| \tilde R_t^{1,N} | &\leq&  \frac{C_1 \norm{f}_L}{N} \sum_{k =0}^{\floor*{Nt}-1} \left|\la f,  \mu^N_{\frac{k}{N}}  - \tilde \mu^N_{\frac{k+1}{N}} \ra\right|  + C_2 N^{-2} \norm{f}_L^2  \notag \\
&\leq&  \frac{C_1 \norm{f}_L}{N} \sum_{k =0}^{\floor*{Nt}-1}  \sup_{c,w \in K} |  f(c,w)  | + C_2 N^{-2} \norm{f}_L \notag \\
&\leq& C_1  \norm{f}_L ^2+  C_2 N^{-2} \norm{f}_L^2  \notag \\
&\leq& C \norm{f}_L^2.\nonumber
\end{eqnarray}

$\tilde R_t^{2,N}$ is the remainder term:
\begin{align}
\tilde R_t^{2,N} &= - 2\sqrt{N} \int_{ \frac{ \floor*{Nt} }{N} }^t \int_{\mathcal{X}\times\mathcal{Y}}   \alpha \la f, \Xi_s^N \ra \big{(} y -  \la c \sigma(w \cdot x),  \mu^N_s \ra \big{)} \la \sigma(w \cdot x) \partial_c f, \mu_t^N \ra \pi(dx,dy) ds \notag \\
&- 2\sqrt{N} \int_{ \frac{ \floor*{Nt} }{N} }^t \int_{\mathcal{X}\times\mathcal{Y}}   \alpha  \la f, \Xi_s^N \ra \big{(} y -  \la c \sigma(w \cdot x),  \mu^N_t \ra \big{)} \la c \sigma'(w \cdot x) x \cdot \nabla_w f, \mu_t^N \ra \pi(dx,dy) ds\nonumber
\end{align}

Using Young's inequality, compactness of $\mathcal{X} \times \mathcal{Y}$, and the bound (\ref{UniformBoundfromLLNpaper}),
\begin{align}
| \tilde R_t^{2,N}  |  &\leq C  \int_{ \frac{ \floor*{Nt} }{N} }^t \int_{\mathcal{X}\times\mathcal{Y}}    \la f, \Xi_s^N \ra^{2} \pi(dx,dy) ds+ C N  \int_{ \frac{ \floor*{Nt} }{N} }^t \int_{\mathcal{X}\times\mathcal{Y}} \big{(} y -  \la c \sigma(w \cdot x),  \mu^N_s \ra \big{)}^{2} \la \sigma(w \cdot x) \partial_c f, \mu_t^N \ra^{2} \pi(dx,dy) ds \notag \\
&\quad+ C N \int_{ \frac{ \floor*{Nt} }{N} }^t \int_{\mathcal{X}\times\mathcal{Y}} \big{(} y -  \la c \sigma(w \cdot x),  \mu^N_t \ra \big{)}^{2} \la c \sigma'(w \cdot x) x \cdot \nabla_w f, \mu_t^N \ra^{2} \pi(dx,dy) ds\notag\\
&\leq C_{1}  \int_{ \frac{ \floor*{Nt} }{N} }^t   \la f, \Xi_s^N \ra^{2} ds+ C_{2} N  \int_{ \frac{ \floor*{Nt} }{N} }^t  \sum_{|\alpha| = 1 } \sup_{c,w \in K} |D^{\alpha} f(c,w)  |^{2} ds\notag\\
&\leq C_1 \int_{ \frac{ \floor*{Nt} }{N} }^t  \la f, \Xi_s^N \ra^2 ds + C_2  N \int_{ \frac{ \floor*{Nt} }{N} }^t  \norm{f}_L^2 ds \notag \\
&\leq  C_1 \int_{0}^t  \la f, \Xi_s^N \ra^2 ds + C_2 \norm{f}_L^2.\nonumber
\end{align}

Hence, we have obtained that
\begin{align}
|\tilde R_t^{1,N}|+| \tilde R_t^{2,N}  |  &\leq   C_1 \int_{0}^t  \la f, \Xi_s^N \ra^2 ds + C_2 \norm{f}_L^2.\nonumber
\end{align}
which is (\ref{Eq:BoundRTerms}). We then notice that
\begin{eqnarray}
&\phantom{.}&  \sum_{k =0}^{\floor*{Nt}-1} \mathbb{E} \bigg{[}   \sqrt{N} \la f, \Xi^N_{\frac{k+1}{N}^{-}} \ra  \la f, M^{1,N}_k + M^{2,N}_k \ra \bigg{]} \notag \\
&=& N  \sum_{k =0}^{\floor*{Nt}-1} \mathbb{E} \bigg{[}     \la f,  \mu_{ \frac{k+1}{N}^{-} }^N - \tilde \mu_{ \frac{k+1}{N}}^N  \ra  \la f, M^{1,N}_k + M^{2,N}_k \ra  \bigg{]}  \notag \\
  &=& N  \sum_{k =0}^{\floor*{Nt}-1} \mathbb{E} \bigg{[}    \la f,  \nu_{ k}^N  \ra \mathbb{E} \bigg{[}   \la f, M^{1,N}_k + M^{2,N}_k \ra \bigg{|} \mathcal{F}_k^{N}  \bigg{]} \bigg{]} \notag \\
  & &- N  \sum_{k =0}^{\floor*{Nt}-1} \mathbb{E} \bigg{[} \mathbb{E} \bigg{[}     \la f, \tilde  \mu_{ \frac{k+1}{N} }^N  \ra \mathbb{E} \bigg{[}   \la f, M^{1,N}_k + M^{2,N}_k \ra \bigg{|} \mathcal{F}_k^{N}  \bigg{]} \bigg{|} \mathcal{F}_0^{N} \bigg{]} \bigg{]}  \notag \\
  &=& 0,
\end{eqnarray}
where 
 $\mathcal{F}_k^N$ is the $\sigma-$algebra generated by $(c^{i}_{0},w^{i}_{0})_{i=1}^{N}$ and $(x_{j}, y_{j})_{j=0}^{k-1}$. In the fourth line we use the conditional independence of $\la f, M^{1,N}_k+  M^{2,N}_k\ra$ and $\tilde \mu^N_{ \frac{k+1}{N}}$ given the initial values $\{ w^i_0 , c^i_0 \}_{i=1}^N$. Also, since $\mu_t^N$ only changes at discrete times due to jumps, $\mu_{ \frac{k+1}{N}^{-} }^N = \nu_{k}^N$.

We have also used the fact that the conditional expectation
\begin{eqnarray}
\mathbb{E} \bigg{[}   \la f, M^{1,N}_k \ra \bigg{|} \mathcal{F}_k^{N}  \bigg{]}  &=&  \mathbb{E} \bigg{[}  \frac{1}{N} \alpha \big{(} y_k -  \la c \sigma(w \cdot x_k),  \nu^N_k \ra \big{)} \la \sigma(w \cdot x_k) \partial_c f, \nu_k^N \ra   - D^{1,N}_k \bigg{|} \mathcal{F}_k^{N} \bigg{]} \notag \\
&=&  \frac{\alpha}{N^2} \sum_{i=1}^N   \mathbb{E} \bigg{[} \bigg{(} \big{(} y_k -  \la c \sigma(w \cdot x_k),  \nu^N_k \ra \big{)}  \sigma(w^i_k \cdot x_k) \partial_c f(c^i_k, w^i_k) \notag \\
& &-  \int_{\mathcal{X}\times\mathcal{Y}}  \big{(} y -  \la c \sigma(w \cdot x),  \nu^N_k \ra \big{)}  \sigma(w^i_k \cdot x) \partial_c f(c^i_k, w^i_k) \pi(dx,dy)  \bigg{)} \bigg{|} \mathcal{F}_k^{N} \bigg{]} \notag \\
&=& 0.\nonumber
\end{eqnarray}
Similarly, $\mathbb{E} \bigg{[}   \la f, M^{2,N}_k \ra \bigg{|} \mathcal{F}_k^{N}  \bigg{]} = 0$.

 Now we can treat the term $\mathbb{E} \bigg{[}  \sum_{k =0}^{\floor*{Nt}-1} \bigg{(} \big{(} \la f, \Xi^N_{\frac{k+1}{N}^{-}}   + \sqrt{N} M^{1,N}_k + \sqrt{N} M^{2,N}_k \ra\big{)}^2  - \la f, \Xi^N_{ \frac{k+1}{N}^{-}} \ra^2  \bigg{)} \bigg{]}$ from (\ref{Ito}) and get
\begin{eqnarray}
&\phantom{.}& \mathbb{E} \bigg{[}  \sum_{k =0}^{\floor*{Nt}-1} \bigg{(} \big{(} \la f, \Xi^N_{\frac{k+1}{N}^{-}}   + \sqrt{N} M^{1,N}_k + \sqrt{N} M^{2,N}_k \ra\big{)}^2  - \la f, \Xi^N_{ \frac{k+1}{N}^{-}} \ra^2  \bigg{)} \bigg{]}  \notag \\
&=& \mathbb{E} \bigg{[}  \sum_{k =0}^{\floor*{Nt}-1} \bigg{(}  2\sqrt{N} \la f, \Xi^N_{\frac{k+1}{N}^{-}} \ra  \la f, M^{1,N}_k + M^{2,N}_k \ra  + N \la f, M^{1,N}_k + M^{2,N}_k \ra^2 \bigg{)}   \bigg{]} \notag \\
&=&  \mathbb{E} \bigg{[} N \sum_{k =0}^{\floor*{Nt}-1} \la f, M^{1,N}_k + M^{2,N}_k \ra ^2   \bigg{]}  \notag \\
&=&\alpha^2 \mathbb{E} \bigg{[} \frac{1}{N}  \sum_{k =0}^{\floor*{Nt}-1}  \bigg{(}   \big{(} y_k -  \la c \sigma(w \cdot x_k),  \nu^N_k \ra \big{)} \la \sigma(w \cdot x_k) \partial_c f, \nu_k^N \ra  \notag \\
& &+ \big{(} y_k -  \la c \sigma(w \cdot x_k),  \nu^N_k \ra \big{)} \la c  \sigma'(w \cdot x_k) x \cdot \nabla_w f, \nu_k^N \ra \notag \\
& &- \int_{\mathcal{X}\times\mathcal{Y}}   \big{(} y -  \la c \sigma(w \cdot x),  \nu^N_k \ra \big{)} \la \sigma(w \cdot x) \partial_c f, \nu_k^N \ra \pi(dx,dy)  \notag \\
& &- \int_{\mathcal{X}\times\mathcal{Y}}   \big{(} y -  \la c \sigma(w \cdot x),  \nu^N_k \ra \big{)} \la c  \sigma'(w \cdot x) x \cdot \nabla_w f, \nu_k^N \ra  \pi(dx,dy)   \bigg{)}^2      \notag \\
&<& C  \big{(} \sum_{|\alpha| =1 } \sup_{c,w \in K} |D^{\alpha} f(c,w)  | \big{)}^2 \leq C \norm{f}_L^2.\nonumber
\end{eqnarray}
which is (\ref{Mbound}) concluding the proof of the lemma.
\end{proof}
\section{Auxiliary lemmas}\label{A:Appendix1}

\begin{lemma} \label{LemmaBoundPsi}
Let $0\leq k\leq J$. If $\Psi \in C_0^{\infty}(\Theta)$, $g \in C^{\infty}_0( \Theta)$, then, there exists a  constant $C<\infty$ such that
\begin{eqnarray}
\int_{\Theta} D^k \bigg{[}  g \frac{\partial \Psi}{\partial w} \bigg{]} D^k \Psi dc dw \leq C \norm{ \Psi}_J^2.
\end{eqnarray}
\end{lemma}

\begin{proof}[Proof of Lemma \ref{LemmaBoundPsi}]
We prove the statement for $d =1$. The algebra for $d>1$ is similar, albeit more tedious.  Let $k = k_1 + k_2$ with $k_1,k_2\geq 0$ arbitrarily chosen.
\begin{eqnarray}
 \int_{\Theta} D^k \bigg{[}  g \frac{\partial \Psi}{\partial w} \bigg{]} D^k \Psi dc dw
&=& \int_{\Theta} \frac{\partial^k}{\partial c^{k_1} \partial w^{k_2}} \bigg{[}  g \frac{\partial \Psi}{\partial w} \bigg{]} D^k \Psi dc dw  \notag \\
&=& \sum_{ \substack{ \alpha_1 + \alpha_2 = k+1, \alpha_2 \leq k \\ i_1 + i_2 = k_1, \\ j_1 + j_2 = k_2 }} \int_{\Theta} \frac{\partial^{\alpha_1} g}{\partial c^{i_1} \partial w^{j_1}} \frac{\partial^{\alpha_2} \Psi }{ \partial c^{i_2} \partial w^{j_2}} D^k \Psi dc dw \notag \\
&+&  \int_{\Theta} g   \frac{\partial}{\partial w} \bigg{[} \frac{\partial^k \Psi}{\partial c^{k_1} \partial w^{k_2}} \bigg{]} D^k \Psi dc dw
\label{SubStackExpression}
\end{eqnarray}

Since $g \in C^{\infty}_{0}(\bar \Theta)$ and using Young's inequality,
\begin{eqnarray}
 \sum_{ \substack{ \alpha_1 + \alpha_2 = k+1, \alpha_2 \leq k \\ i_1 + i_2 = k_1, \\ j_1 + j_2 = k_2 }} \int_{\Theta} \frac{\partial^{\alpha_1} g}{\partial c^{i_1} \partial w^{j_1}} \frac{\partial^{\alpha_2} \Psi }{ \partial c^{i_2} \partial w^{j_2}} D^k \Psi dc dw
&\leq& C \sum_{ \substack{ \alpha_1 + \alpha_2 = k+1, \alpha_2 \leq k \\ i_1 + i_2 = k_1, \\ j_1 + j_2 = k_2 }} \int_{\Theta}  \bigg{|} \frac{\partial^{\alpha_2} \Psi }{ \partial c^{i_2} \partial w^{j_2}} \bigg{|} \bigg{|} D^k \Psi \bigg{|} dc dw \notag \\
&\leq& C \sum_{ \substack{ \alpha_1 + \alpha_2 = k+1, \alpha_2 \leq k \\ i_1 + i_2 = k_1, \\ j_1 + j_2 = k_2 }} \int_{\Theta}  \bigg{(}  \bigg{|} \frac{\partial^{\alpha_2} \Psi }{ \partial c^{i_2} \partial w^{j_2}} \bigg{|}^2 + \bigg{|} D^k \Psi \bigg{|}^2 \bigg{)} dc dw \notag \\
&\leq& C  \norm{ \Psi}_J^2.
\label{BoundSubStack}
\end{eqnarray}


Therefore, we have
\begin{eqnarray}
 \int_{\Theta} D^k \bigg{[}  g \frac{\partial \Psi}{\partial w} \bigg{]} D^k \Psi dc dw
&\leq& C_1 \norm{ \Psi}_J^2 +   \int_{\Theta} g   \frac{\partial}{\partial w} \bigg{[} \frac{\partial^k \Psi}{\partial c^{k_1} \partial w^{k_2}} \bigg{]} \frac{\partial^k \Psi}{\partial c^{k_1} \partial w^{k_2}} dc dw \notag \\
&=& C_2 \norm{ \Psi}_J^2 -   \int_{\Theta} \frac{\partial^k \Psi}{\partial c^{k_1} \partial w^{k_2}}\frac{\partial}{\partial w} \bigg{[} g \frac{\partial^k \Psi}{\partial c^{k_1} \partial w^{k_2}} \bigg{]} dc dw.\notag
\end{eqnarray}

The inequality on line 2 follows from the bound (\ref{BoundSubStack}).  The third line follows from integration by parts and the fact that $g \in C_0^{\infty}(\Theta)$.

We next consider the other term
\begin{eqnarray}
\frac{\partial}{\partial w} \bigg{[} g \frac{\partial^k \Psi}{\partial c^{k_1} \partial w^{k_2}} \bigg{]} &=& g D^k \bigg{[} \frac{\partial \Psi}{\partial w} \bigg{]} + \frac{\partial g}{\partial w}(c,w) D^k \Psi \notag \\
&=& D^k \bigg{[} g   \frac{\partial \Psi}{\partial w} \bigg{]}   + \frac{\partial g}{\partial w}(c,w) D^k \Psi-  \sum_{ \substack{ \alpha_1 + \alpha_2 = k+1, \alpha_2 \leq k \\ i_1 + i_2 = k_1, \\ j_1 + j_2 = k_2 }}  \frac{\partial^{\alpha_1} g}{\partial c^{i_1} \partial w^{j_1}} \frac{\partial^{\alpha_2} \Psi }{ \partial c^{i_2} \partial w^{j_2}},\notag
\end{eqnarray}
where the last term is from (\ref{SubStackExpression}). Now, by applying the same approach as in (\ref{BoundSubStack}), i.e. using Young's inequality and $g \in C_{0}^{\infty}(\bar \Theta)$, we have the bound
\begin{eqnarray}
 \int_{\Theta} D^k \bigg{[}  g \frac{\partial \Psi}{\partial w} \bigg{]} D^k \Psi dc dw
&\leq& C_2 \norm{ \Psi}_J^2 -   \int_{\Theta} \frac{\partial^k \Psi}{\partial c^{k_1} \partial w^{k_2}}\frac{\partial}{\partial w} \bigg{[} g \frac{\partial^k \Psi}{\partial c^{k_1} \partial w^{k_2}} \bigg{]} dc dw   \notag \\
&\leq& C \norm{\Psi}_J^2 -  \int_{\Theta} D^k \bigg{[}  g \frac{\partial \Psi}{\partial w} \bigg{]} D^k \Psi dc dw.\notag
\end{eqnarray}

Rearranging, we have that there is a constant $C<\infty$ (different than above)
\begin{eqnarray}
 \int_{\Theta} D^k \bigg{[}  g(c,w) \frac{\partial \Psi}{\partial w} \bigg{]} D^k \Psi dc dw  \leq C \norm{ \Psi}_J^2.\notag
\end{eqnarray}

\end{proof}

\end{document}